\documentclass[12pt,letterpaper]{article}
\usepackage{amsthm,amssymb}
\usepackage{amsfonts}
\makeatletter
\usepackage[dvips]{color}
\usepackage{colordvi,multicol}

\newtheorem{thm}{\textbf Theorem}[section]
\newtheorem{lem}[thm]{\textbf Lemma}
\newtheorem{rem}[thm]{\textbf Remark}

\newtheorem{prop}[thm]{\textbf Proposition}
\newtheorem{defin}[thm]{\textbf Definition}
\newtheorem{assum}[thm]{\textbf Assumption}

\newtheorem{exa}[thm]{\textbf Example}
\newtheorem{condition}[thm]{\textbf Condition}
\newcommand{\be}{\begin{eqnarray*}}
\newcommand{\ee}{\end{eqnarray*}}

\begin{document}

\title{\bf Fukushima type decomposition for semi-Dirichlet forms\footnote{ %acknowledgment of support etc. if any
We are grateful to the support of 973 project (2011CB808000),
NCMIS, NSFC (11021161), and NSERC (Grant No. 311945-2013). }}
  \author{Zhi-Ming Ma, Wei Sun and Li-Fei Wang}
   \date{}
\maketitle

\noindent \textbf{Abstract.} We present a Fukushima type
decomposition in the setting of general quasi-regular
semi-Dirichlet forms. The decomposition is then employed to give a
transformation formula for martingale additive functionals.
Applications of the results to some concrete examples of
semi-Dirichlet forms are given at the end of the paper. We discuss
also the uniqueness question about Doob-Meyer decomposition on
optional sets of interval type. \vskip 0.3cm \noindent \textbf{Key
words and phrases.} Fukushima type decomposition,
 quasi-regular semi-Dirichlet forms, stochastic sets of interval type,
transformation formula for martingale additive functionals.\vskip
0.3cm \noindent \textbf{2000 Mathematics Subject Classification.}
Primary 31C25; Secondary 60J25.

%%%%%%%%%%%%%%%%%%%%%%%%%%%%%%%%%%%%%%%%%%%%%%%%%%%%%%%%%%%%%%%%%%%%%%%
\section[short title]{Introduction}

The celebrated Fukushima's decomposition and related
transformation rules play the roles of Doob-Meyer decomposition
and It$\hat{\rm o}$'s formula in the framework of Dirichlet forms.
They have been used to investigate the properties of a large class
of stochastic processes that are not semi-martingales such as
additive functionals of Brownian motion which are not necessarily
of bounded variation  (cf. e.g. \cite{TZ}, \cite{CF} and
references therein). Fukushima's decomposition was originally
established for regular symmetric Dirichlet forms (cf. \cite{Fu79}
and \cite[Theorem 5.2.2]{Fu94}) and then extended to the
non-symmetric and quasi-regular cases (cf. \cite[Theorem
5.1.3]{oshima} and \cite[Theorem VI.2.5]{MR92}).  Suppose that
$({\cal E},D({\cal E}))$ is a quasi-regular Dirichlet form on
$L^2(E;m)$ with associated Markov process $((X_t)_{t\ge 0},
(P_x)_{x\in E_{\Delta}})$ (we refer the reader to
\cite{Fu94,MR92,MR95} for notations and terminologies of this
paper). If $u\in D({\cal E})$, then Fukushima's decomposition
tells us that there exist a unique martingale additive functional
(MAF in short) $M^{[u]}$ of finite energy and a continuous
additive functional $N^{[u]}$ of zero energy such that
\begin{equation}\label{Doob}
\tilde{u}(X_t)-\tilde{u}(X_0)=M^{[u]}_t+N^{[u]}_t.
\end{equation}
Hereafter $\tilde{u}$ denotes an ${\cal E}$-quasi-continuous
$m$-version of $u$.

Compared with Dirichlet form, semi-Dirichlet form is a more
general framework arising from various applications. In the
viewpoint of applications, and also by the interests of the theory
its own, it is natural to ask if we can extend Fukushima's
decomposition from the setting of Dirichlet forms to that of
semi-Dirichlet forms. For example, do we have Fukushima's
decomposition for the following simple local semi-Dirichlet form?
\begin{eqnarray*}
\mathcal{E}(u,v)=\int^{1}_{0}u'v'dx+\int^{1}_{0}\sqrt{x}u'vdx,\ \
u,v\in D(\mathcal{E}):=H^{1,2}_{0}(0,1).
\end{eqnarray*}
Note that the assumption of the existence of dual Markov process
plays a crucial role in Fukushima's decomposition for Dirichlet
forms. In fact, without that assumption, the usual definition of
energy of AFs is questionable. Here we would like to point out
that although Fukushima's decomposition was even considered for
generalized Dirichlet forms (cf. \cite{T} and \cite{Sun}), which
is a more general framework than semi-Dirichlet forms (see
\cite{WSta}), up to now Fukushima's decomposition for generalized
Dirichlet forms has only been given under the additional
assumption that their dual forms are also sub-Markovian. For a
quasi-regular semi-Dirichlet form $({\cal E},D({\cal E}))$, we may
use the semi-$h$ transform method to associate  $({\cal E},D({\cal
E}))$ with a sub-Markovian dual form (cf. \cite{HMS}). However,
without imposing further assumptions, we cannot expect to obtain
Fukushima's decomposition for general $u\in D({\cal E})$; we can
only expect to obtain the decomposition (\ref{Doob}) for functions
$u$ in the domain of the generator of $({\cal E},D({\cal E}))$,
which is just the classical Doob-Meyer decomposition.

To our knowledge, the paper
 \cite{MMS} appears to be the first publication on
the Fukushima type decomposition  in the semi-Dirichlet forms
setting without assuming that the dual form is sub-Markovian. In
that paper the authors introduced a condition of local control
(cf. Condition \ref{assum2} below) and under the condition they
obtained the Fukushima type decomposition for $u\in D({\cal
E})_{loc}$ where $({\cal E},D({\cal E}))$ is  a local
semi-Dirichlet form. The main method employed in \cite{MMS} is the
localization and pasting technique. For a non-local semi-Dirichlet
form, the jump part of $M^{[u]}$ is in general not locally
consistent, which causes some extra difficulty in implementing the
localization and pasting technique. Afterwards, one of the authors
of the present paper investigated further in \cite{W} on the
Fukushima type decomposition for general quasi-regular
semi-Dirichlet forms. Motivated by some idea of Kuwae
\cite{kuwae2010} and employing also the localization and pasting
technique, he obtained the Fukushima type decomposition for $u\in
D({\cal E})_{loc}$ under a suitable condition (S) (see Theorem
\ref{thm3.2} below). Meanwhile Professor Oshima sent us a
manuscript of his new book \cite{oshima13}, in which he proved
Fukushima's decomposition for $u\in D({\cal E})_{b}$ in the
setting of regular semi-Dirichlet forms satisfying his condition
$(\mathcal{E}.5)$. The main techniques employed by Oshima
 in developing Fukushima's decomposition are the weak sense
 energy and his genius auxiliary bilinear form, different from the localization and pasting technique employed in \cite{MMS} and \cite{W}.

In this paper we shall report and develop further the Fukushima
type decomposition based on \cite{W}, and discuss some related
topics. Let $({\cal E},D({\cal E}))$ be a quasi-regular
semi-Dirichlet form which is not necessarily local. We show that
under a suitable assumption (i.e. Assumption \ref{assum1} below),
a function $u\in {D(\mathcal{E})}_{loc}$ admits a Fukushima type
decomposition if and only if it satisfies Condition (S), and the
decomposition is unique. Roughly speaking, here $u$ admits a
Fukushima type decomposition means that
$$
\tilde{u}(X_{t})-\tilde{u}(X_{0})=M^{[u]}_{t}+N^{[u]}_{t},
$$
where $M^{[u]}$ is a locally square integrable MAF on the set
$I(\zeta):=[\![0,\zeta[\![\cup[\![\zeta_i]\!],$ with $\zeta$ being
the lifetime of $X$ and  $\zeta_i$ the totally inaccessible part
of $\zeta;$ and $N^{[u]}$ is a local AF which is continuous and
has zero quadratic variation on $I(\zeta)$. For details see
Theorem \ref{thm3.2} below. It is worth to point out that
Assumption \ref{assum1} mentioned above is weaker than the
condition of local control in \cite{MMS} and the condition (${\cal
E}.5$) in \cite{oshima13}. We are very grateful to Professor
Oshima for sending us his new book \cite{oshima13}. The condition
(${\cal E}.5$) in \cite{oshima13} stimulated us to formulate
Assumption \ref{assum1}.

The reader might notice that in the above description we used
$I(\zeta)$ instead of $[\![0,\zeta[\![,$ the latter is customarily
used in the literature. The reason of this variation is that we
discovered that the decomposition on $I(\zeta)$ is unique, but it
may fail to be unique on $[\![0,\zeta[\![$.  This difference is
essentially due to the fact  that $I(\zeta)$ is a  predictable set
of interval type while $[\![0,\zeta[\![$ is not necessarily
predictable. This discovery  exposes not only an oversight in the
previous paper \cite{MMS}, but also similar oversights in the
literature e.g. \cite{Chen} and \cite{kuwae2010}. The oversight
may be traced back even to Theorem 8.26 of the book \cite{HWY},
which exposes a question about the uniqueness of Doob-Meyer
decomposition on optional sets of interval type. We shall discuss
this question in detail in Section \ref{sec:uniqueness} below.

The rest of the paper is organized as follows. In Section
\ref{Sec:Fukushima}, we present a general  Fukushima type
decomposition for semi-Dirichlet forms. We divide it into two
subsections. In Subsection \ref{1} we present basic settings and
statement of the theorem, and provide some discussions and remarks
about the theorem. In Subsection \ref{2}, we give the proof of the
theorem. In Section \ref{sec:uniqueness}, we discuss in detail the
question about the uniqueness of Doob-Meyer decomposition on
optional sets of interval type. In Section \ref{sec:transform}, we
give a transformation formula for MAFs based on the Fukushima type
decomposition. In Section \ref{Sec:example}, we apply our results
to two concrete examples of semi-Dirichlet forms appearing in
recent papers.

%%%%%%%%%%%%%%%%%%%%%%%%%%%%%%%%%%%%%%%%%%%%%%%%%%%%%%%%%%%%%%%%%%%%%%%%%%%%%%%%%%%%%%%%%%%%%%%%%
\section{Fukushima type decomposition}\label{Sec:Fukushima} \setcounter{equation}{0}
\subsection{Statement of the theorem and discussions }\label{1}
The basic setting of this paper is the same as that in \cite{MMS}
with some necessary modifications, e.g., $({\cal E},D({\cal E}))$
in this paper is not assumed to be local. To fix the notations and
also for the convenience of the reader, below we restate our
setting of which some contents are taken from \cite{MMS}. Let $E$
be a metrizable Lusin space and $m$ a $\sigma$-finite positive
measure on its Borel $\sigma$-algebra ${\cal B}(E)$. We consider a
quasi-regular semi-Dirichlet form  $({\cal E},D({\cal E}))$ on
$L^2(E;m)$. Hereafter for notations and terminologies related to
quasi-regular semi-Dirichlet forms we refer to \cite{MR95}. Denote
by $(T_{t})_{t\geq0}$ and $(G_{\alpha})_{\alpha\geq0}$ (resp.
$(\hat{T}_{t})_{t\geq0}$ and $(\hat{G}_{\alpha})_{\alpha\geq0}$)
the semigroup and resolvent (resp. co-semigroup and co-resolvent)
associated with $({\cal E},D({\cal E}))$. Let ${\bf
M}=(\Omega,{\cal F},({\cal F}_t)_{t\ge 0}, (X_t)_{t\ge
0},(P_x)_{x\in E_{\Delta}})$ be an $m$-tight special standard
process which is properly associated with $({\cal E},D({\cal E}))$
in the sense that $P_tf$ is an ${\cal E}$-quasi-continuous
$m$-version of $T_tf$ for all $f\in {\cal B}_b(E)\cap L^2(E;m)$
and all $t>0$, where $(P_t)_{t\ge 0}$ denotes the semigroup
associated with ${\bf M}$ (cf. \cite[Theorem 3.8]{MR95}).

Similar to the symmetric case, in the semi-Dirichlet forms setting
there is also a one-to-one correspondence between the family of
all equivalent classes of positive continuous additive functionals
and the family $S$ of smooth measures. The contents below
concerning positive continuous additive functionals and $S$ are
taken from \cite{MMS}.  We remark that the reader can now find
more detailed descriptions and discussions in \cite{oshima13} on
the potential theory of semi-Dirichlet forms including the
correspondence between positive continuous additive functionals
and smooth measures.

Recall that a positive measure $\mu$ on $(E, {\cal B}(E))$ is
called {\it smooth} (w.r.t. $({\cal E},D({\cal E}))$), denoted by
$\mu \in S,$  if $\mu(N)=0$ for each ${\cal E}$-exceptional set
$N\in {\cal B}(E)$ and there exists an ${\cal E}$-nest $\{F_k\}$
of compact subsets of $E$ such that
$$
\mu(F_k)<\infty\ {\rm for\ all}\ k\in\mathbb{N}.
$$
A family $(A_t)_{t\ge 0}$ of functions on $\Omega$ is called an
{\it additive functional} (AF in short) of ${\bf M}$ if:

(i) $A_t$ is ${\cal F}_t$-measurable for all $t\ge 0$.

(ii) There exists a defining set $\Lambda\in{\cal F}$ and an
exceptional set $N\subset E$ which is ${\cal E}$-exceptional such
that $P_x[\Lambda]=1$ for all $x\in E\backslash N$,
$\theta_t(\Lambda)\subset\Lambda$ for all $t>0$ and for each
$\omega\in \Lambda$, $t\rightarrow A_t(\omega)$ is right
continuous on $(0,\infty)$ and has left limits on
$(0,\zeta(\omega))$, $A_0(\omega)=0$, $|A_t(\omega)|<\infty$ for
$t<\zeta(\omega)$, $A_t(\omega)=A_{\zeta}(\omega)$ for
$t\ge\zeta(\omega)$, and
\begin{eqnarray}\label{additivity}
A_{t+s}(\omega)=A_t(\omega)+A_s(\theta_t\omega),~~~~\forall~
s,t\ge 0.
\end{eqnarray}
Hereafter $\zeta$ denotes the lifetime of $X:=(X_t)_{t\geq 0}$.

Two AFs $A=(A_t)_{t\ge 0}$ and $B=(B_t)_{t\ge 0}$ are said to be
equivalent, denoted by $A=B,$ if they have a common defining set
$\Lambda$ and a common exceptional set $N$ such that
$A_t(\omega)=B_t(\omega)$ for all $\omega\in\Lambda$ and $t\ge 0$.
An AF $A=(A_t)_{t\ge 0}$ is called a continuous AF (CAF in short)
if $t\rightarrow A_t(\omega)$ is continuous on $(0,\infty).$ It is
called a positive CAF (PCAF in short) if $A_t(\omega)\ge 0$ for
all $t\ge 0$, $\omega\in\Lambda$.

\begin{lem}\label{lem:Revuz}(cf. \cite[Theorem A.8]{MMS}, see also \cite[Section 4.1]{oshima13})
Let $A$ be a PCAF. Then there exists a unique $\mu\in S$, which is
referred to as the {\it Revuz measure} of $A$ and is denoted by
$\mu_A,$ such that:

For any $\gamma$-co-excessive function $g$ $(\gamma\geq0)$ in
$D({\cal E})$ and $f\in \mathcal{B}^{+}(E)$,
$$
 \lim_{t\downarrow0}\frac{1}{t}E_{g\cdot
 m}((fA)_{t})=<f\cdot\mu,\tilde{g}>.
$$
Conversely, let $\mu\in S$, then there exists a unique (up to the
equivalence) PCAF $A$ such that $\mu =\mu_A.$
\end{lem}

Throughout this paper, we fix a function $\phi\in L^{2}(E;m)$ with
 $0<\phi\leq1$ $m{\textrm{-}}a.e.$ and set $h=G_{1}\phi$,
 $\hat{h}=\hat{G}_{1}\phi$.  Denote $\tau_B:=\inf\{t>0\,|\, X_t\notin B\}$ for $B\subset E$.
Let $V$ be a quasi-open subset of $E$. We denote by
$X^V=(X^V_t)_{t\ge 0}$ the part process of $X$ on $V$ and denote
by $(\mathcal{E}^{{V}},D(\mathcal{E})_{{V}})$ the part form of
$(\mathcal{E},D(\mathcal{E}))$ on $L^2(V;m)$. It is known that
$X^V$ is a standard process and
$(\mathcal{E}^{{V}},D(\mathcal{E})_{{V}})$ is a quasi-regular
semi-Dirichlet form (cf. \cite{kuwae}). Denote by
$({T}_{t}^{V})_{t\ge 0}$, $(\hat{T}_{t}^{V})_{t\ge 0}$,
$({G}_{\alpha}^{V})_{\alpha\ge 0}$ and
$(\hat{G}_{\alpha}^{V})_{\alpha\ge 0}$ the semigroup,
co-semigroup, resolvent and co-resolvent associated with
$(\mathcal{E}^{{V}},D(\mathcal{E})_{{V}})$, respectively. One can
check that $\hat{h}|_{V}$ is 1-co-excessive w.r.t.
$(\mathcal{E}^{{V}},D(\mathcal{E})_{{V}})$. Define
$\bar{h}^V:=\hat{h}|_{V}\wedge\hat{G}^{{V}}_{1}\phi$. Then
$\bar{h}^V\in D(\mathcal{E})_{{V}}$ and $\bar{h}^V$ is
 1-co-excessive. Denote $D(\mathcal{E})_{{V},b}:={\cal B}_b(E)\cap D(\mathcal{E})_{{V}}$.

 For an AF $A=(A_t)_{t\ge 0}$ of $X^V$, we define
$$
e^V(A):=\lim_{t\downarrow0}{1\over{2t}}E_{\bar{h}^V\cdot
m}(A_{t}^2)
$$
whenever the limit exists in $[0,\infty]$. Define
 \begin{eqnarray*}
\dot{\mathcal{M}}^V
 &:=&\{M\,|\, M\ \mbox{is an AF of}\ X^{V},\ E_x(M^{2}_t)<\infty, E_x(M_t)=0\\
 &&\ \ \ \ \mbox{for}\ {\rm all}\ t\ge0\ {\rm and}\ {\cal
E}{\textrm{-}q.e.}\ x\in V, e^V(M)<\infty\},
 \end{eqnarray*}
  \begin{eqnarray*}
 \mathcal{N}^V_{c}
 &:=&\{N\,|\, N\ \mbox{is a CAF of}\ X^{V},E_x(|N_t|)<\infty\ \mbox{for}\ {\rm all}\ t\ge0\\
   &&\ \ \ \ {\rm and}\ {\cal
E}{\textrm{-}q.e.}\ x\in V, e^V(N)=0\},
   \end{eqnarray*}
\begin{eqnarray*}
\Theta&:=&\{\{{V_n}\}\,|\, {V_n}\ \mbox{is}\ {\cal
E}{\textrm{-}}\mbox{quasi} {\textrm{-}}\mbox{open},\ {V_n}\subset
V_{n+1}\ {\cal E}
{\textrm{-}q.e.} \\
         &&\ \ \ \ \ \ \ \ \ \ \ \forall~n\in \mathbb{N},\
         \mbox{and}\ E=\cup_{n=1}^{\infty}{V_n}\ {\cal E}{\textrm{-}q.e.}\},
\end{eqnarray*}
and
\begin{eqnarray*}
{D(\mathcal{E})}_{loc}
&:=&\{u\,|\,\exists\ \{V_n\}\in\Theta\ \mbox{and}\ \{u_n\}\subset D(\mathcal{E})\nonumber\\
  &&\ \ \ \ \ \ \ \mbox{such that }\ u=u_n\ m{\textrm{-}a.e.}\ \mbox{on}\ V_n,
  ~\forall~n\in \mathbb{N}\}.
\end{eqnarray*}
In what follows we shall employ the notion of local AFs introduced
in \cite{Fu94} as follows.

\begin{defin} (cf. \cite[page 271]{Fu94})\label{localAF}
A family $A=(A_t)_{t\ge 0}$ of functions on $\Omega$ is called a
{\it local AF} of ${\bf M},$ if $A$ satisfies all the requirements
for an AF as stated in above (i) and (ii), except that the
additivity property (\ref{additivity}) is
 required only for $s,t\ge 0$ with
$t+s<\zeta(\omega)$.
\end{defin}

Two local AFs $A^{(1)}$, $A^{(2)}$ are said to be equivalent if
for
 ${\cal E}{\textrm{-}q.e.}\ x\in E$, it holds that
$$
P_x(A^{(1)}_t=A^{(2)}_t;t<\zeta)=P_x(t<\zeta), ~~\forall~t\geq0.
$$
We now define
\begin{eqnarray*}
\dot{\mathcal{M}}_{loc}
 &:=&\{M\,|\, M\ \mbox{is a local AF of}\ {\bf M},\ \exists\ \{V_n\},\{E_n\}\in\Theta\ {\rm and}\ \{M^n\,|\,M^n\in\dot{\mathcal{M}}^{V_n}\}\\
 &&\ \ \ \ \ \ \ \ \mbox{such that}\ E_n\subset V_n,\ M_{t\wedge\tau_{E_n}}=M^{n}_{t\wedge\tau_{E_n}},\ t\ge0,\ n\in\mathbb{N} \}
 \end{eqnarray*}
 and
 \begin{eqnarray*}
{\mathcal{L}}_{c}&:=&\{N\,|\, N\ \mbox{is a local AF of}\ {\bf M}\ ,\ \exists\ \{E_n\}\in\Theta\ \mbox{such that}\ t\rightarrow N_{t\wedge\tau_{E_n}}\\
&&\ \ \ \ \ \ \mbox{  is continuous and of zero quadratic
variation},\  n\in\mathbb{N} \}.
 \end{eqnarray*}

We use $\zeta_i$ to denote the totally inaccessible part of
$\zeta$, by which we  mean that $\zeta_i$ is an $\{{\cal
F}_t\}$-stopping time and is the totally inaccessible part of
$\zeta$ w.r.t. $P_{x}$ for ${\cal E}{\textrm{-}q.e.}\ x\in E$. In
Section 3 below we shall give a proof for the existence and
uniqueness of such $\zeta_i$, where the uniqueness is in the sense
of $P_{x}{\textrm{-}a.s.}$ for ${\cal E}{\textrm{-}q.e.}\ x\in E$.
Write $I(\zeta):=[\![0,\zeta[\![\cup[\![\zeta_i]\!]$. We can show
that there exists a $\{V_n\}\in\Theta$ such that for any
$\{U_n\}\in\Theta$, $I(\zeta)=\cup_n [\![0,\tau_{V_n\cap U_n}]\!]$
$ P_{x}{\textrm{-}a.s.}\ \ {\rm for}\ {\cal E}{\textrm{-}q.e.}\
x\in E$ (see Proposition \ref{prop00} below). Therefore $I(\zeta)$
is a predictable set of interval type
 (cf. \cite[Theorem 8.18]{HWY}).
 In this paper a local AF $M$ is called a locally square integrable MAF on $I(\zeta)$,
 denoted by $M\in {\mathcal{M}}^{I(\zeta)}_{loc},$  if $M\in ({\mathcal{M}}^2_{loc})^{I(\zeta)}$ in the sense of \cite[Definition 8.19]{HWY}.

Denote by $J(dx,dy)$ the jump measure of $({\cal E},D({\cal E}))$
(cf. \cite{HC06}). Let $(N(x,dy),H_s)$ be a L\'{e}vy system of
$X$. Then we have $J(dy,dx)=N(x,dy)\mu_H(dx)$.

We put the following assumption:
\begin{assum}\label{assum1}
There exist $\{{V_n}\}\in \Theta$ and locally bounded function
$\{C_n\}$ on $\mathbb{R}$ such that for each $n\in\mathbb{N}$, if
$u,v\in D(\mathcal{E})_{{V_n},b}$ then $uv\in D(\mathcal{E})$ and
\begin{eqnarray*}
\mathcal{E}(uv,uv)\leq C_n(\|u\|_{\infty}+\|v\|_{\infty})({\cal
E}_1(u,u)+{\cal E}_1(v,v)).
\end{eqnarray*}
\end{assum}

Now we can state the main theorem of this section.
\begin{thm}\label{thm3.2} Suppose that
$(\mathcal{E},D(\mathcal{E}))$ is a quasi-regular semi-Dirichlet
form on $L^{2}(E;m)$ satisfying Assumption \ref{assum1}. Then for
$u\in {D(\mathcal{E})}_{loc}$ the following two assertions are
equivalent to each other.

(i) $u$ admits a Fukushima type decomposition. That is, there
exist $M^{[u]}\in{\mathcal{M}}^{I(\zeta)}_{loc}$ and $N^{[u]}\in
{\mathcal{L}}_{c}$ such that
\begin{eqnarray}\label{new3}
\tilde{u}(X_{t})-\tilde{u}(X_{0})=M^{[u]}_{t}+N^{[u]}_{t},\ \
t\ge0,\ \ P_{x}{\textrm{-}a.s.}\ \ {\rm for}\ {\cal
E}{\textrm{-}q.e.}\ x\in E.
\end{eqnarray}

(ii) $u$ satisfies Condition (S) specified below.
\begin{eqnarray*}
(S):\  \  \  \mu_u(dx):=\int_E(\tilde u(x)-\tilde u(y))^2J(dy,dx)\
\mbox{is a smooth measure}.
\end{eqnarray*}
Moreover, if $u$ satisfies Condition (S), then the decomposition
(\ref{new3}) is unique up to the equivalence of local AFs, and the
continuous part of $M^{[u]}$ belongs to $\dot{\mathcal{M}}_{loc}$.
\end{thm}
The proof of Theorem \ref{thm3.2} will be given in the next
subsection.  In the remainder of this subsection we provide some
remarks and discussions about the theorem.

In \cite{MMS}, the authors obtained a Fukushima type decomposition
for $u\in D(\mathcal{E})_{loc}$  where $({\cal E},D({\cal E}))$ is
a local quasi-regular Dirichlet form satisfying the condition of
local control as stated below.

\begin{condition}\label{assum2}
There exists $\{{V_n}\}\in \Theta$ such that for each
$n\in\mathbb{N}$ there exist a  Dirichlet form
$(\eta^{(n)},D(\eta^{(n)}))$ on $L^2(V_n;m)$ and a constant
$C_n>1$ satisfying $D(\eta^{(n)})= D(\mathcal{E})_{{V_n}}$ and for
any $u\in D(\mathcal{E})_{{V_n}}$,
\begin{eqnarray*}
\frac{1}{C_n}\eta^{(n)}_1(u,u)\leq\mathcal{E}_1(u,u)\leq
C_n\eta^{(n)}_1(u,u).
\end{eqnarray*}
\end{condition}
It is clear that Assumption \ref{assum1} is more general than
Condition \ref{assum2}. Hence we have the following remark.
\begin{rem}
Theorem \ref{thm3.2} extends the corresponding result of
\cite{MMS}.
\end{rem}
In \cite{oshima13}, Oshima discussed various topics of regular
semi-Dirichlet forms under his condition $(\mathcal{E}.5)$. In
particular,
 he proved in Theorem 5.1.5 a weak sense of Fukushima's decomposition for
$u\in D({\cal E})_{b}$. Below is the condition $(\mathcal{E}.5)$
of \cite{oshima13} stated in our context. \vskip 0.2cm \noindent
{\bf Condition $(\mathcal{E}.5)$.}  If $u,v \in D(\mathcal{E})$
and $w \in L^2(E;m)$  satisfy $|w(x)-w(y)|\leq |u(x)-u(y)| +
|v(x)-v(y)|~\mbox{and}~|w(x)|\leq |u(x)|+ |v(x)|$ for any $x,y\in
E$, then $w \in D(\mathcal{E})$ and $|\mathcal{E}(w,w)|\leq
K(\mathcal{E}_{1}(u,u)+\mathcal{E}_{1}(v,v))$ for some $K$
depending on $\|u\|_{\infty}$ and $\|v\|_{\infty}.$
 \vskip 0.2cm
It is easy to see that Condition $(\mathcal{E}.5)$ implies the
following condition.
\begin{condition}\label{assum3}
There exists a locally bounded function $C$ on $\mathbb{R}$ such
that if $u,v\in D(\mathcal{E})_b,$ then $uv\in D(\mathcal{E})$ and
\begin{eqnarray}\label{lkj}
\mathcal{E}(uv,uv)\leq C(\|u\|_{\infty}+\|v\|_{\infty})({\cal
E}_1(u,u)+{\cal E}_1(v,v)).
\end{eqnarray}
\end{condition}
\begin{prop}\label{propO}
Suppose that  $({\cal E},D({\cal E}))$ satisfies Condition
\ref{assum3}, then any $u\in D(\mathcal{E})_b$ satisfies Condition
(S), and hence admits a Fukushima type decomposition.
\end{prop}
\begin{proof}
Since Condition \ref{assum3} is a special case of Assumption
\ref{assum1},  hence by Theorem \ref{thm3.2} we need only to check
that any $u\in D(\mathcal{E})_b$ satisfies Condition (S). By the
quasi-homeomorphism method (cf. \cite{CMR} or \cite[Theorem
3.8]{HC06}), without loss of generality below we assume that
$({\cal E},D({\cal E}))$ is a regular semi-Dirichlet form.
  Let $\{E_n\}$ be a sequence of relatively compact open sets such that $E=\cup_nE_n$
  and $\{u_n\}\subset D(\mathcal{E})\cap C_0(E)$ satisfying $u_n=1$ on $E_n$ for each $n\in\mathbb{N}$. We choose a sequence of relatively
compact open sets $G_l\uparrow E$ and a sequence of numbers
$\delta_l\downarrow 0$ such that the set $\Gamma_l:=\{(x,y)\in
G_l\times G_l\,|\,|\rho(x,y)\ge \delta_l\}$ is a continuous set
w.r.t. $J$ for every $l\in \mathbb{N}$, where $\rho$ is the metric
of $E$. For $\beta>0$, let $\sigma_{\beta}$ be the unique positive
Radon measures on $E\times E$ satisfying
$$
(\beta G_{\beta}f,g)=\int_{E\times
E}f(x)g(y)\sigma_{\beta}(dxdy),\ \ \forall f,g\in
D(\mathcal{E})\cap C_0(E).
$$

Let $u\in D(\mathcal{E})\cap C_0(E)$. Then, for each $n\in
\mathbb{N}$,
\begin{eqnarray}\label{addlk}
&&\int_{E_n}\int_{E}(u(x)-u(y))^{2}N(x,dy)\mu_{H}(dx)\nonumber\\
&&\ \ \ \leq\int_{E}\int_{E}u_n(x)(u(x)-u(y))^{2}N(x,dy)\mu_{H}(dx)\nonumber\\
&&\ \ \ \leq\lim_{l\rightarrow\infty}\int\int_{\Gamma_l}(u(x)-u(y))^2u_n(y)J(dx,dy)\nonumber\\
&&\ \ \ =\lim_{l\rightarrow\infty}\lim_{\beta\rightarrow\infty}\frac{\beta}{2}\int\int_{\Gamma_l}(u(x)- u(y))^2u_n(y)\sigma_{\beta}(dx,dy)\nonumber\\
&&\ \ \ \leq\lim_{\beta\rightarrow\infty}\frac{\beta}{2}\int_E\int_{E}(u(x)-u(y))^2u_n(y)\sigma_{\beta}(dx,dy)\nonumber\\
&&\ \ \ \leq\lim_{\beta\rightarrow\infty}\frac{\beta}{2}\{(\beta G_{\beta}1_E,u^2u_n)-2(\beta G_{\beta}u,uu_n)+(\beta G_{\beta}u^2,u_n)\}\nonumber\\
&&\ \ \ \leq\lim_{\beta\rightarrow\infty}\{\beta(u-\beta G_{\beta}u,uu_n)-\frac{\beta}{2}(u^2-\beta G_{\beta}u^2,u_n)\}\nonumber\\
&&\ \ \ ={\cal E}(u,uu_n)-\frac{1}{2}{\cal E}(u^2,u_n),
\end{eqnarray}
which implies that $u$ satisfies Condition (S).

For general $u\in D(\mathcal{E})_b$, we may select a sequence of
functions $\{u_k\}\subset D(\mathcal{E})\cap C_0(E)$ such that
$u_k\rightarrow u$ w.r.t. the $\tilde {\cal E}^{1/2}_1$-norm as
$k\rightarrow\infty$ and $\|u_k\|_{\infty}\le \|u\|_{\infty}$ for
$k\in \mathbb{N}$. Then by (\ref{lkj}), (\ref{addlk}) and Fatou's
lemma, we can show that $\int_{E_n}\int_{E}(\tilde u(x)-\tilde
u(y))^{2}N(x,dy)\mu_{H}(dx)<\infty.$ Hence $u$ satisfies Condition
(S), which completes the proof.
\end{proof}
\begin{rem}
Proposition \ref{propO} shows that Theorem \ref{thm3.2} is an
extension of \cite[Theorem 5.1.5]{oshima13}.
\end{rem}
We would like to point out that the methods of \cite{oshima13} in
developing Fukushima's decomposition are  different from ours. In
the next subsection we shall see that Theorem \ref{thm3.2} is
proved by the localization and pasting technique. The main
techniques employed by Oshima in developing his Theorem 5.1.5 are
the weak sense energy and the genius auxiliary bilinear form
invented in \cite{oshima13}. We take this opportunity to thank
Professor Oshima for sending us his manuscript \cite{oshima13}.
The condition (${\cal E}.5$) in \cite{oshima13} stimulated us to
formulate Assumption  \ref{assum1}.

\begin{rem}
Theorem \ref{thm3.2} extends the corresponding results of
  \cite[Theorem 5.5.1]{Fu94} and
\cite[Theorem 4.2]{kuwae2010} from the symmetric case to the
semi-Dirichlet form case. \end{rem} Note that for a symmetric
Dirichlet form $({\cal E},D({\cal E}))$, Assumption \ref{assum1}
is satisfied automatically. Also, $u\in D({\cal E})_{loc}$
satisfies Condition (S) trivially if $({\cal E},D({\cal E}))$ is
local. When $({\cal E},D({\cal E}))$ is non-local, Condition (S)
is necessary even in the symmetric case. In developing stochastic
analysis with Nakao's integral, Kuwae obtained  in
\cite{kuwae2010} a generalized Fukushima decomposition in the
symmetric case for a subclass of $D({\cal E})_{loc},$ which is
equivalent to impose Condition (S) for $u\in D({\cal E})_{loc}.$
In this paper when dealing with purely discontinuous part of
$M^{[u]},$  we adopted some idea from \cite{kuwae2010} without
making use of Nakao's integral. One of the authors of this paper
has joint work with others extending Nakao's integral to
non-symmetric Dirichlet forms (cf. \cite{Chen0}). We feel that
Nakao's integral can also be extended to semi-Dirichlet forms.
\begin{rem} In Theorem \ref{thm3.2} if  we use ${\mathcal{M}}^{[\![0,\zeta[\![}_{loc}$ instead of ${\mathcal{M}}^{I(\zeta)}_{loc}$,
then the uniqueness of the decomposition may fail to be true.
\end{rem}
We shall discuss the above remark and related topics in detail in
Section \ref{sec:uniqueness} below.

\subsection{Proof of the theorem}\label{2}
Before proving Theorem \ref{thm3.2}, we prepare some lemmas.

We fix a $\{{V_n}\}\in\Theta$ satisfying Assumption \ref{assum1}.
Without loss of generality, we assume that $\widetilde{\hat{h}}$
is bounded on each ${V_n}$, otherwise we may replace ${V_n}$ by
${V_n}\cap\{\widetilde{\hat{h}}< n \}$. To simplify notations, we
write
$$\bar{h}_{n}:=\bar{h}^{V_n}.
$$
\begin{lem}\label{thm3.6}(\cite[Lemma
2.6]{MMS}) Let $u\in D(\mathcal{E})_{V_n,b}$. Then there exist
unique $M^{n,[u]}\in\dot{\mathcal{M}}^{V_n}$ and $N^{n,[u]}\in
\mathcal{N}^{V_n}_{c}$ such that ${\rm for}\ {\cal
E}{\textrm{-}q.e.}\ x\in V_n$,
$$
\tilde{u}(X^{{V_n}}_{t})-\tilde{u}(X^{{V_n}}_{0})=M^{n,[u]}_{t}+N^{n,[u]}_{t},\
\ t\ge 0,\ \ P_{x}{\textrm{-}a.s.}
$$
\end{lem}

Lemma \ref{thm3.6} has been given in \cite{MMS} under Assumption
\ref{assum2} and the additional assumption that $({\cal E},D({\cal
E}))$ is local; however, it can be easily extended to general
semi-Dirichlet forms under Assumption \ref{assum1} with the
similar proof.

We now fix a $u\in {D(\mathcal{E})}_{loc}$ satisfying Condition
(S). Then there exist $\{V^1_n\}\in \Theta$ and $\{u_n\}\subset
D(\mathcal{E})$ such that $u=u_n$ $m{\textrm{-}a.e.}$ on $V^1_n$.
By \cite[Proposition 3.6]{MR95}, we may assume without loss of
generality that  each $u_n$ is ${\cal E}$-quasi-continuous. By
\cite[Proposition 2.16]{MR95}, there exists an $\mathcal{E}$-nest
$\{F_{n}^2\}$ of compact subsets of $E$ such that $\{u_n\}\subset
C\{F_{n}^2\}$. Denote by $V^{2}_{n}$ the finely interior of
$F^2_n$. Then $\{V^{2}_{n}\}\in\Theta$. Since $u$ satisfies
Condition (S), there exists an $\mathcal{E}$-nest $\{F_{n}^3\}$ of
compact subsets of $E$ such that $\mu_u(F_{n}^3)<\infty$. Denote
by $V^{3}_{n}$ the finely interior of $F^3_n$. Since the killing
measure $K(dx)=N(x,\Delta)\mu_H(dx)$ is a smooth measure, there
exists an $\mathcal{E}$-nest $\{F_{n}^4\}$ of compact subsets of
$E$ such that $K(F_{n}^4)<\infty$. Denote by $V^{4}_{n}$ the
finely interior of $F^4_n$. Define ${V'_n}=V^1_n\cap V^2_{n}\cap
V^3_{n}\cap V^4_{n}$. Then $\{{V'_n}\}\in\Theta$, each $u_n$ is
bounded on ${V'_n}$, and
\begin{eqnarray*}
&&\int_{V'_n}\int_{E_{\Delta}}(\tilde u(x)-\tilde u(y))^2N(x,dy)\mu_H(dx)\\
&&\ \ \ \ \ =\int_{V'_n}\int_E(\tilde u(x)-\tilde u(y))^2J(dy,dx)+\int_{V'_n}\tilde u^2(x)K(dx)\\
&&\ \ \ \ \ <\infty.
\end{eqnarray*}
To simplify notation, we still use $V_n$ to denote $V_n\cap V'_n$.

For $n\in\mathbb{N}$, we define $E_{n}=\{x\in
E\,|\,{\widetilde{h_n}}(x)>{1\over n}\}$, where
$h_n:=G_{1}^{{V_n}}\phi$. Then $\{E_{n}\}\in\Theta$ satisfying
$\overline{E}_n^{\mathcal{E}}\subset E_{n+1}\ {\cal
E}{\textrm{-}q.e.}$ and $E_{n}\subset {V_n}\ {\cal
E}{\textrm{-}q.e.}$ for each $n\in \mathbb{N}$ (cf. \cite[Lemma
3.8]{kuwae}). Here $\overline{E}_n^{\mathcal{E}}$ denotes the
${\cal E}$-quasi-closure of $E_n$. Define
$f_{n}=n\widetilde{h_n}\wedge1$. Then $f_{n}=1$ on $E_{n}$ and
$f_{n}=0$ on $V^c_{n}$. Since $f_{n}$ is a 1-excessive function of
$(\mathcal{E}^{V_n},D(\mathcal{E})_{V_n})$ and $f_{n}\leq
n\widetilde{h_n}\in D(\mathcal{E})_{V_n}$, hence $f_{n}\in
D(\mathcal{E})_{V_n}$ by \cite[Remark 3.4(ii)]{MR1995}. Denote by
$Q_n$ the bound of $|u_n|$ on $V_n$. Then  $u_{n}f_{n}=((-Q_n)\vee
u_n\wedge Q_n)f_n\in D(\mathcal{E})_{V_n,b}$. For $n\in
\mathbb{N}$, we denote by $\{\mathcal{F}^n_{t}\}$ the minimum
completed admissible filtration of $X^{{V_n}}$. For $n<l$,
$\mathcal{F}^n_{t}\subset \mathcal{F}^l_{t}\subset
\mathcal{F}_{t}$. Since $E_n\subset V_n$, $\tau_{E_n}$ is an
$\{\mathcal{F}^{n}_{t}\}$-stopping time.

\begin{lem}\label{KKK}(\cite[Lemma
25.3]{K}) For any optional time $T$ and predictable process $Y$,
the random variable $Y_T1_{(T<\infty)}\in\mathcal{F}_{T-}$.
\end{lem}

Hereafter for a martingale $M$, we denote by $M^c$ and $M^d$ its
continuous part and purely discontinuous part, respectively.

\begin{lem}\label{lem3.7} For $n<l$, we have
 $M^{n,[u_nf_n],c}_{t\wedge\tau_{E_n}}=M^{l,[u_lf_l],c}_{t\wedge\tau_{E_n}}$
, $t\ge0$, $P_{x}{\textrm{-}a.s.}\ \ {\rm for}\ {\cal
E}{\textrm{-}q.e.}\ x\in  V_n$.
\end{lem}

\begin{proof} Let $n<l$. Since
$M^{n,[u_nf_n]}\in\dot{\mathcal{M}}^{V_n}$, $M^{n,[u_nf_n]}$ is an
$\{\mathcal{F}^{n}_{t}\}$-martingale by the Markov property. Since
$\tau_{E_n}$ is an $\{\mathcal{F}^{n}_{t}\}$-stopping time,
$\{M^{n,[u_nf_n]}_{_{t\wedge\tau_{E_{n}}}}\}$ is an
$\{\mathcal{F}^{n}_{t\wedge\tau_{E_{n}}}\}$-martingale. Denote
$\Upsilon^n_t=\sigma\{X^{{V_n}}_{s\wedge\tau_{E_n}}\,|\,0\leq
s\leq t\}$. Then $\{M^{n,[u_nf_n],c}_{_{t\wedge\tau_{E_{n}}}}\}$
is a $\{\Upsilon^n_{t}\}$-martingale. Denote
$\Upsilon^{n,l}_t=\sigma\{X^{{V_l}}_{s\wedge\tau_{E_n}}\,|\,0\leq
s\leq t\}$. Similarly, we can show that
$\{M^{l,[u_nf_n],c}_{_{t\wedge\tau_{E_{n}}}}\}$ is a
$\{\Upsilon^{n,l}_{t}\}$-martingale. Since
\begin{equation}\label{eqadd}
X^{V_l}_{s}
 =X_{s} =X_{s}^{{V_n}},\ \ s<\tau_{E_n},\ P_{x}{\textrm{-}a.s.}\ \ {\rm for}\ {\cal
E}{\textrm{-}q.e.}\ x\in V_n,
\end{equation}
we find that $\Upsilon^n_{t-}=\Upsilon^{n,l}_{t-}$. Hence
$\{M^{l,[u_nf_n],c}_{_{t\wedge\tau_{E_{n}}}}\}\in\Upsilon^{n,l}_{t-}$
by Lemma \ref{KKK} and therefore
$\{M^{l,[u_nf_n],c}_{_{t\wedge\tau_{E_{n}}}}\}$ is a
$\{\Upsilon^n_t\}$-martingale. Moreover,
$N^{l,[u_nf_n]}_{t\wedge\tau_{E_{n}}}\in\Upsilon^{n,l}_{t-}=\Upsilon^{n}_{t-}\subset\mathcal{F}^n_{t\wedge\tau_{E_{n}}}$.

Let $N\in \mathcal{N}^{V_j}_c$ for some $j\in \mathbb{N}$. Then,
for any $T>0$,
\begin{eqnarray*}
\sum_{k=1}^{[rT]}E_{\bar{h}_j\cdot
m}[(N_{\frac{k+1}{r}}-N_{\frac{k}{r}})^{2}]
&\leq&\sum_{k=1}^{[rT]}e^{T}(E_{\cdot}(N_{\frac{1}{r}}^{2}),e^{-\frac{k}{r}}\hat{T}^{V_j}_{\frac{k}{r}}\bar{h}_j)\\
&\leq&\sum_{k=1}^{[rT]}e^{T}(E_{\cdot}(N_{\frac{1}{r}}^{2}),\bar{h}_j)\\
&\leq&rTe^{T}E_{\bar{h}_j\cdot
m}(N_{\frac{1}{r}}^{2})\rightarrow0\ \ \ \mbox{as} \ \
r\rightarrow\infty.
\end{eqnarray*}
Hence
$$\sum_{k=1}^{[rT]}(N_{\frac{k+1}{r}}-N_{\frac{k}{r}})^{2}\rightarrow0\ \ {\rm in}\ \ P_{m}\ \ {\rm as}\ r\rightarrow\infty,$$ which implies that
the quadratic variation process of $N$ w.r.t. $P_m$ is 0.
Therefore, the quadratic variation processes of
$\{N^{l,[u_nf_n]}_{t\wedge\tau_{E_{n}}}\}$ and
$\{N^{n,[u_nf_n]}_{t\wedge\tau_{E_{n}}}\}$ w.r.t. $P_m$ are 0.

By \cite[Proposition 3.3]{kuwae},
 $(\widehat{\hat{G}_{1}\phi})^{1}_{V^{c}_{n}}=\hat{G}_{1}\phi-\hat{G}_{1}^{{V_n}}\phi$.
Since $V^{c}_{n}\supset V^{c}_l$,
$(\widehat{\hat{G}_{1}\phi})^{1}_{V^{c}_{n}}\geq
(\widehat{\hat{G}_{1}\phi})^{1}_{V^{c}_l}$. Then
$\hat{G}_{1}^{{V_n}}\phi\leq \hat{G}_{1}^{V_l}\phi$ and thus
\begin{equation}\label{new2}\bar{h}_{n}\le\bar{h}_{l}.
\end{equation}
Therefore
\begin{equation}\label{new1} e^{V_n}(A)\le e^{V_l}(A)
\end{equation}
for any AF $A=(A_t)_{t\ge 0}$ of $X^{{V_n}}$.

By (\ref{eqadd}), we find that ${\rm for}\ {\cal
E}{\textrm{-}q.e.}\ x\in V_n$,
\begin{eqnarray*}
&
&M^{n,[u_nf_n],c}_{t\wedge\tau_{E_{n}}}+M^{n,[u_nf_n],d}_{t\wedge\tau_{E_{n}}}+N^{n,[u_nf_n]}_{t\wedge\tau_{E_{n}}}\\
&&\ \ \ \ =\widetilde{u_nf_n}(X^{{V_n}}_{t\wedge\tau_{E_n}})-\widetilde{u_nf_n}(X^{{V_n}}_0)\\
&&\ \ \ \ =\widetilde{u_nf_n}(X^{V_l}_{t\wedge\tau_{E_n}})-\widetilde{u_nf_n}(X^{V_l}_0)\\
&&\ \ \ \
=M^{l,[u_nf_{n}],c}_{t\wedge\tau_{E_{n}}}+M^{l,[u_nf_n],d}_{t\wedge\tau_{E_{n}}}+N^{l,[u_nf_{n}]}_{t\wedge\tau_{E_{n}}},
\ \ P_x-a.s.
\end{eqnarray*}
Then $\{M^{n,[u_nf_n],d}_{t\wedge\tau_{E_n}}\}\in\Upsilon^n_t$,
and $\{M^{n,[u_nf_n]}_{t\wedge\tau_{E_{n}}}\}$ and
$\{M^{l,[u_nf_n]}_{t\wedge\tau_{E_{n}}}\}$ are
$\{\Upsilon^n_t\}$-martingales. Hence
$M^{n,[u_nf_n],c}_{t\wedge\tau_{E_{n}}}=M^{l,[u_nf_n],c}_{t\wedge\tau_{E_{n}}}$
 and $N^{n,[u_nf_n]}_{t\wedge\tau_{E_{n}}}=N^{l,[u_nf_n]}_{t\wedge\tau_{E_{n}}}$, $P_{x}{\textrm{-}a.s.}\ \ {\rm for}\ m{\textrm{-}a.e.}\ x\in V_n$.
  This implies that
$E_m(<M^{n,[u_nf_n],c}_{\cdot\wedge\tau_{E_{n}}}-M^{l,[u_nf_n],c}_{\cdot\wedge\tau_{E_{n}}}>_
t)=0$, $\forall t\ge 0$. By Theorem \cite[Theorem 5.8(i)]{MMS}, we
find that
$M^{n,[u_nf_n],c}_{t\wedge\tau_{E_{n}}}=M^{l,[u_nf_n],c}_{t\wedge\tau_{E_{n}}}$,
$\forall t\ge 0$,
 $P_{x}{\textrm{-}a.s.}\ \ {\rm for}\ {\cal
E}{\textrm{-}q.e.}\ x\in V_n$.

Since $u_nf_n=u_lf_{l}=u$ on $E_n$, similar to \cite[Lemma
2.4]{kuwae2010}, we can show that
$M^{l,[u_nf_{n}],c}_{t}=M^{l,[u_lf_l],c}_{t}$ when
$t<\tau_{E_{n}}$, $P_{x}{\textrm{-}a.s.}\ \ {\rm for}\ {\cal
E}{\textrm{-}q.e.}\ x\in V_l$. Then
$M^{l,[u_nf_n],c}_{t\wedge\tau_{E_n}}=M^{l,[u_lf_l],c}_{t\wedge\tau_{E_n}}$,
$t\ge0$, $P_{x}{\textrm{-}a.s.}\ \ {\rm for}\ {\cal
E}{\textrm{-}q.e.}\ x\in  V_l$. Therefore
$M^{n,[u_nf_n],c}_{t\wedge\tau_{E_n}}=M^{l,[u_lf_l],c}_{t\wedge\tau_{E_n}}$
, $t\ge0$, $P_{x}{\textrm{-}a.s.}\ \ {\rm for}\ {\cal
E}{\textrm{-}q.e.}\ x\in  V_n$.
\end{proof}

\vskip 0.2cm

\noindent\textbf{Proof of Theorem \ref{thm3.2}}\ \ (a) Suppose
that $u$ satisfies Condition (S). We shall show that $u$ admits
the Fukushima type decomposition (\ref{new3}).

We define
$M^{[u],c}_{t\wedge\tau_{E_n}}:=\lim_{l\rightarrow\infty}M^{l,[uf_l],c}_{t\wedge\tau_{E_n}}$
and $M^{[u],c}_t:=0$ for $t>\zeta$  if there exists some $n$ such
that $\tau_{E_n}=\zeta$ and $\zeta<\infty$; or $M^{[u],c}_t:=0$
for $t\ge\zeta$, otherwise. By Lemma \ref{lem3.7}, $M^{[u],c}$ is
well defined. Define
$M^n_t:=M^{n+1,[uf_{n+1}],c}_{t\wedge\tau_{E_n}}$ for $t\ge 0$ and
$n\in\mathbb{N}$. Then
$M^{[u],c}_{t\wedge\tau_{E_n}}=M^{n}_{t\wedge\tau_{E_n}}$
$P_{x}{\textrm{-}a.s.}\ \ {\rm for}\ {\cal E}{\textrm{-}q.e.}\
x\in  V_{n+1}$ by Lemma \ref{lem3.7}. Since
$\overline{E}_n^{\mathcal{E}}\subset E_{n+1}\subset V_{n+1}\ {\cal
E}{\textrm{-}q.e.}$ implies that $P_x(\tau_{E_n}=0)=1$ for
$x\notin V_{n+1}$,
$M^{[u],c}_{t\wedge\tau_{E_n}}=M^{n}_{t\wedge\tau_{E_n}}$
$P_{x}{\textrm{-}a.s.}\ \ {\rm for}\ {\cal E}{\textrm{-}q.e.}\
x\in  E$. Similar to (\ref{new2}) and (\ref{new1}), we can show
that  $e^{V_n}(M^{n})\le e^{V_{n+1}}(M^{n})$ for each
$n\in\mathbb{N}$. Then $M^{n}\in \dot{\mathcal{M}}^{V_n}$ and
hence $M^{[u],c}\in \dot{\mathcal{M}}_{loc}$.

Next we show that $M^{n}$ is also an
$\{\mathcal{F}_{t}\}$-martingale. In fact, by the fact that
$\tau_{E_n}$ is an $\{\mathcal{F}^{n+1}_{t}\}$-stopping time, we
find that $I_{\tau_{E_n}\le s}$ is ${\cal
F}^{n+1}_{s\wedge\tau_{E_n}}$-measurable for any $s\ge 0$. Let
$0\le s_1<\cdots<s_k\le s<t$ and $g\in{\cal B}_b(\mathbb{R}^k)$.
Then, we obtain by the fact $M^{n+1,[uf_{n+1}],c}\in
\dot{\mathcal{M}}^{V_{n+1}}$ that ${\rm for}\ {\cal
E}{\textrm{-}q.e.}\ x\in  V_{n+1}$,
\begin{eqnarray*}
& &\int_{\Omega} M^n_tg(X_{s_1},\dots,X_{s_k})dP_x\\&&\ \ \ \
=\int_{\tau_{E_n}\le s}M^n_tg(X_{s_1},\dots,X_{s_k})dP_x
+\int_{\tau_{E_n}>s}M^n_tg(X_{s_1},\dots,X_{s_k})dP_x\\
&&\ \ \ \ =\int_{\tau_{E_n}\le s}M^n_sg(X_{s_1},\dots,X_{s_k})dP_x\\
&&\ \ \ \ \ \ \ \ +\int_{\Omega}
M^{n+1,[uf_{n+1}],c}_{t\wedge\tau_{E_n}}g(X^{V_{n+1}}_{s_1\wedge\tau_{E_n}},
\dots,X^{V_{n+1}}_{s_k\wedge\tau_{E_n}})
I_{\tau_{E_n}> s}dP_x\\
&&\ \ \ \ =\int_{\tau_{E_n}\le s}M^n_sg(X_{s_1},\dots,X_{s_k})dP_x\\
&&\ \ \ \ \ \ \ \ +\int_{\Omega}
M^{n+1,[uf_{n+1}],c}_{s\wedge\tau_{E_n}}g(X^{V_{n+1}}_{s_1\wedge\tau_{E_n}},
\dots,X^{V_{n+1}}_{s_k\wedge\tau_{E_n}})
I_{\tau_{E_n}> s}dP_x\\
&&\ \ \ \ =\int_{\tau_{E_n}\le s}M^n_sg(X_{s_1},\dots,X_{s_k})dP_x
+\int_{\tau_{E_n}>s}M^n_sg(X_{s_1},\dots,X_{s_k})dP_x\\
&&\ \ \ \ =\int_{\Omega} M^n_sg(X_{s_1},\dots,X_{s_k})dP_x.
\end{eqnarray*}
Obviously, the equality holds for $x\notin  V_{n+1}$. Hence
$M^{n}$ is an $\{\mathcal{F}_{t}\}$-martingale. By Proposition
\ref{prop00}  below, $\cup_n [\![0,\tau_{E_n}]\!]\supseteq
I(\zeta)$ $ P_{x}{\textrm{-}a.s.}\ \ {\rm for}\ {\cal
E}{\textrm{-}q.e.}\ x\in E$. Therefore
$M^{[u],c}\in{\mathcal{M}}^{I(\zeta)}_{loc}$.

We define $\phi(x,y)=\tilde u(y)-\tilde u(x)$,
$\phi_{l}(x,y)=(\tilde u(y)-\tilde u(x))1_{\{|\tilde u(x)-\tilde
u(y)|>\frac{1}{l}\}}$, and
\begin{eqnarray*}
M^{l}_{t}&:=&\sum_{0<s\leq
t}\phi_{l}(X_{s-},X_{s})-\int^{t}_{0}\int_{E_{\Delta}}\phi_{l}(X_{s},y)N(X_{s},dy)dH_{s}
 \end{eqnarray*}
 for $l\in\mathbb{N}$. Denote $T^{l}_{m}:=\inf\{t>0\,|\,|M^{l}_{t}|\geq m\}$ for $m\in\mathbb{N}$.
 Then, $\{T^{l}_{m}\}$ is an $\{{\cal F}_t\}$-stopping time and
\begin{eqnarray*}
|M^{l}_{t\wedge T^{l}_{m}\wedge\tau_{E_{n}}}|&\leq&|M^{l}_{t\wedge
T^{l}_{m}\wedge\tau_{E_{n}}-}|+|\phi(X_{t\wedge
T^{l}_{m}\wedge\tau_{E_{n}}-},X_{t\wedge
T^{l}_{m}\wedge\tau_{E_{n}}})|\\
&\leq& m+|\phi(X_{t\wedge T^{l}_{m}\wedge\tau_{E_{n}}-},X_{t\wedge
T^{l}_{m}\wedge\tau_{E_{n}}})|.
\end{eqnarray*}

We define (cf. \cite[Theorem 5.3]{MMS})
$$
\hat{S}^{*}_{00}:=\{\mu\in S_{0}\,|\,\hat{U}_{1}\mu\leq
c\hat{G}_{1}\phi\ \mbox{for some constant}\ \ c>0\}.
$$
Let $\nu\in S^{*}_{00}$ satisfying $\nu(E)<\infty$. Then, by
\cite[Lemma 5.9]{MMS}, we get
\begin{eqnarray*}
E_{\nu}[(M^{l}_{t\wedge T^{l}_{m}\wedge\tau_{E_{n}}})^{2}]&\leq& 2m^{2}\nu(E)+2E_{\nu}\left[\sum_{s\leq t\wedge\tau_{E_n}}\phi^{2}(X_{s-},X_s)\right]\\
&=&2m^{2}\nu(E)+2E_{\nu}\left[\int^{t\wedge\tau_{E_n}}_{0}\int_{E_{\Delta}}\phi^{2}(X_s,y)N(X_s,dy)dH_{s}\right]\\
&\leq&
2m^{2}\nu(E)+2C_{\nu}(1+t)\int_{E_n}\widetilde{\hat{h}}\int_{E_{\Delta}}\phi^{2}(x,y)N(x,dy)\mu_{H}(dx)\\
&<&\infty,
\end{eqnarray*}
where $C_{\nu}$ is a positive constant. Hence, for fixed $n$ and
$m$, $t\rightarrow M^{l}_{t\wedge T^{l}_{m}\wedge\tau_{E_{n}}}$ is
a square integrable purely discontinuous $P_{\nu}$-martingale. By
\cite[Corollary A.3.1]{Fu94}, we find that
$$
(M^{l}_{t\wedge T^{l}_{m}\wedge\tau_{E_{n}}})^{2}-\sum_{s\leq
t}(\Delta M^{l}_{s\wedge
T^{l}_{m}\wedge\tau_{E_{n}}})^{2}=(M^{l}_{t\wedge
T^{l}_{m}\wedge\tau_{E_{n}}})^{2}-\sum_{s\leq t\wedge
T^{l}_{m}\wedge\tau_{E_n}}\phi^{2}_{l}(X_{s-},X_s)
$$
is a $P_{\nu}-$ martingale, which implies that
\begin{eqnarray*}
E_{\nu}[(M^{l}_{t\wedge\tau_{E_{n}}})^{2}]&\leq&\liminf_{m\rightarrow\infty}E_{\nu}[(M^{l}_{t\wedge
T^{l}_{m}\wedge\tau_{E_{n}}})^{2}]\\
&=&\liminf_{m\rightarrow\infty}E_{\nu}\left[\sum_{s\leq t\wedge T^{l}_{m}\wedge\tau_{E_n}}\phi^{2}_{l}(X_{s-},X_s)\right]\\
&=&E_{\nu}\left[\sum_{s\leq t\wedge\tau_{E_n}}\phi^{2}_{l}(X_{s-},X_s)\right]\\
&\leq& E_{\nu}\left[\int^{t\wedge\tau_{E_n}}_{0}\int_{E_{\Delta}}\phi^{2}(X_s,y)N(X_s,dy)dH_{s}\right]\\
&\leq&C_{\nu}(1+t)\int_{E_n}\widetilde{\hat{h}}(x)\int_{E_{\Delta}}\phi^{2}(x,y)N(x,dy)\mu_{H}(dx)\\
%&\leq&C_{\nu}C_{n}(1+t)\int_{E}\widetilde{\hat{h}}\int_{E_{\Delta}}\phi^{2}(x,y)N(x,dy)\mu_{H}(dx)\\
&<&\infty.
\end{eqnarray*}
Thus $\{M^{l}_{t\wedge\tau_{E_n}}\}$ is a
$P_{\nu}$-square-integrable martingale. Since $\{M^{l}_{t\wedge
T^{l}_{m}\wedge\tau_{E_n}}\}^{\infty}_{m=1}$ is
$L^{2}(P_{\nu})-$bounded, by virtue of Banach-Saks theorem, we
obtain that
\begin{eqnarray*}
E_{\nu}[(M^{l}_{t\wedge\tau_{E_n}})^{2}]=E_{\nu}\left[\int^{t\wedge\tau_{E_n}}_{0}\int_{E_{\Delta}}\phi^{2}_{l}(X_s,y)N(X_s,dy)dH_{s}\right].
\end{eqnarray*}

By Doob's maximum inequality, we obtain that for any $\alpha>0$
and $l,k$,
\begin{eqnarray*}
&&P_{\nu}\left(\sup_{0\leq s\leq T}|M^{l}_{s\wedge\tau_{E_n}}-M^{k}_{s\wedge\tau_{E_n}}|>\alpha\right)\\
&&\ \ \
\leq\frac{4C_{\nu}(1+T)}{\alpha^{2}}\int_{E_n}\widetilde{\hat{h}}(x)\int_{E_{\Delta}}(\phi_{l}-\phi_{k})^{2}(x,y)N(x,dy)\mu_{H}(dx).
\end{eqnarray*}
By the diagonal  method, we may select a subsequence
$l_{k}\rightarrow\infty$ such that for each $n$ when $k\geq n$,
\begin{eqnarray*}
\int_{E_n}\widetilde{\hat{h}}(x)\int_{E_{\Delta}}(\phi_{l_{k+1}}-\phi_{l_{k}})^{2}(x,y)N(x,dy)\mu_{H}(dx)\leq\frac{1}{2^{3k}}.
\end{eqnarray*}
Then
\begin{eqnarray*}
P_{\nu}\left(\sup_{0\leq s\leq
T}|M^{l_{k+1}}_{s\wedge\tau_{E_n}}-M^{l_{k}}_{s\wedge\tau_{E_n}}|>\frac{1}{2^{k}}\right)\leq\frac{C_{\nu}(1+T)}{2^{k}}.
\end{eqnarray*}

Define
$\Lambda_{0}^{n}=\{\omega\in\Omega\,|\,M^{l_{k}}_{s\wedge\tau_{E_n}}\mbox{converges
uniformly in $s$ on each finite interval}\}$. Then,
$\Lambda_{0}^{n_1}\supset\Lambda_{0}^{n_2}$ for $n_1\leq n_2$. By
the Borel-Cantelli lemma, we get
\begin{eqnarray*}
P_{\nu}((\Lambda^{n}_{0})^c)=0\ \ \mbox{for} \   \nu\in \hat
S^{*}_{00}\ {\rm with}\ \nu(E)<\infty.
\end{eqnarray*}
Therefore $P_{x}((\Lambda^{n}_{0})^c)=0$ for ${\cal
E}\textrm{-}q.e.\ x\in E$ (cf. \cite[Theorem A.3]{MMS}). Let
$\Gamma_k$ be the defining set of the MAF $M^{l_k}$, denote
$\Gamma=\cap_k\Gamma_k$ and $\Lambda^n=\Lambda^{n}_0\cap\Gamma$.
Then we have $P_{x}((\Lambda^{n})^c)=0$ for ${\cal
E}\textrm{-}q.e.\ x\in E$. For each $\omega\in\Lambda^n$,
$M^{l_k}_{t\wedge\tau_{E_n}}$ converges uniformly in $t$ on each
finite interval and for each $k$,
\begin{eqnarray*}
M^{l_k}_{(t+s)\wedge\tau_{E_n}}=M^{l_k}_{t\wedge\tau_{E_n}}+M^{l_k}_{s\wedge\tau_{E_n}}\circ\theta_{t\wedge\tau_{E_n}},\
\mbox{if} \ 0\leq t,s<\infty.
\end{eqnarray*}
Thus,  $L^{n}$, the limit of
$\{M^{l_{k}}_{s\wedge\tau_{E_n}}\}_{k=1}^{\infty}$, is a
$P_x$-square integrable purely  discontinuous martingale for
${\cal E}\textrm{-}q.e.\ x\in E$ and satisfies:
\begin{eqnarray*}
L^{n}_{(t+s)\wedge\tau_{E_n}}=L^{n}_{t\wedge\tau_{E_n}}+L^{n}_{s\wedge\tau_{E_n}}\circ\theta_{t\wedge\tau_{E_n}},\
\mbox{if} \ 0\leq t,s<\infty.
\end{eqnarray*}
 By the above construction, we
find that
$L^{n_{1}}_{t\wedge\tau_{E_{n_{1}}}}=L^{n_{2}}_{t\wedge\tau_{E_{n_{1}}}}$
for $n_{1}\leq n_{2}$. We define
$M^{[u],d}_{t}=L^{n}_{t},t\leq\tau_{E_{n}}$, and
$M^{[u],d}_{t}=L^{n}_{t},t\geq\zeta$, if for some $n$,
$\tau_{E_n}=\zeta<\infty$; $M^{[u],d}_{t}=0,t\geq\zeta$,
otherwise. Then $M^{[u],d}\in{\mathcal{M}}^{I(\zeta)}_{loc}$,
which gives all the jumps of $\tilde{u}(X_t)-\tilde{u}(X_0)$ on
$I(\zeta)$. Since $\{M^{l}_{t}\}$ is an MAF for each $l$, we find
that $\{M^{[u],d}_{t}\}$ is a local MAF by the uniform convergence
on $I(\zeta)$.

We define
$N^{[u]}_{t\wedge\tau_{E_n}}:=\tilde{u}(X_{t\wedge\tau_{E_n}})-\tilde{u}(X_0)-M^{[u],c}_{t\wedge\tau_{E_n}}-M^{[u],d}_{t\wedge\tau_{E_n}}$
for each $n\in\mathbb{N}$. Then $N^{[u]}$ is a local AF of ${\bf
M}$ and $t\rightarrow N^{[u]}_{t\wedge\tau_{E_n}}$ is continuous.
Now we show that the quadratic variation process of $N^{[u]}$ is
zero and hence $N^{[u]}\in {\mathcal{L}}_{c}$. By Fukushima's
decomposition for part processes, we have that for $k>n$,
\begin{eqnarray*}
\widetilde{u_kf_{k}}(X_{t\wedge\tau_{E_n}})-\widetilde{u_kf_{k}}(X_{0})
&=&\widetilde{u_kf_{k}}(X^{V_{k}}_{t\wedge\tau_{E_n}})-\widetilde{u_kf_{k}}(X^{V_{k}}_{0})\\
&=&M^{k,[u_kf_{k}]}_{t\wedge\tau_{E_n}}+N^{k,[u_kf_{k}]}_{t\wedge\tau_{E_n}}\\
&=&M^{k,[u_kf_{k}],c}_{t\wedge\tau_{E_n}}+M^{k,[u_kf_{k}],d}_{t\wedge\tau_{E_n}}+N^{k,[u_kf_{k}]}_{t\wedge\tau_{E_n}}
\end{eqnarray*}
and
$$\tilde{u}(X_{t\wedge\tau_{E_n}})-\tilde{u}(X_0)=M^{[u],c}_{t\wedge\tau_{E_n}}+M^{[u],d}_{t\wedge\tau_{E_n}}+N^{[u]}_{t\wedge\tau_{E_n}}.
$$
Then
\begin{eqnarray*}
N^{[u]}_{t\wedge\tau_{E_n}}
&=&N^{k,[u_kf_{k}]}_{t\wedge\tau_{E_n}}+M^{k,[u_kf_{k}],d}_{t\wedge\tau_{E_n}}-M^{[u],d}_{t\wedge\tau_{E_n}}+\tilde{u}(X_{t\wedge\tau_{E_n}})
-\widetilde{u_kf_{k}}(X_{t\wedge\tau_{E_n}})\\
&=&N^{k,[u_kf_{k}]}_{t\wedge\tau_{E_n}}+M^{k,[u_kf_{k}],d}_{t\wedge\tau_{E_n}}-M^{[u],d}_{t\wedge\tau_{E_n}}+[\tilde{u}(X_{\tau_{E_n}})
-\widetilde{u_kf_{k}}(X_{\tau_{E_n}})]1_{\{\tau_{E_n}\leq t\}}.
\end{eqnarray*}
Define
$A_{t}:=[\tilde{u}(X_{\tau_{E_n}})-\widetilde{u_kf_{k}}(X_{\tau_{E_n}})]1_{\{\tau_{E_n}\leq
t\}}$. Since both $\{M^{k,[u_kf_{k}],d}_{t\wedge\tau_{E_n}}\}$ and
$\{M^{[u],d}_{t\wedge\tau_{E_n}}\}$ are
$\{\mathcal{F}_{t\wedge\tau_{E_n}}\}$-purely discontinuous
martingales, $\tau_{E_n}$ is an
$\{\mathcal{F}_{t\wedge\tau_{E_n}}\}$-stopping time, and
$\tilde{u}(X_{\zeta})=\widetilde{u_kf_{k}}(X_{\zeta})=0$, we find
that $\{A_{t}\}$ is an adapted quasi-left continuous bounded
variation processes. Denote by $\{A^{p}_{t}\}$ the dual
predictable projection of $A$. Then $\{A^{p}_{t}\}$
 is an adapted continuous bounded variation
processes (cf. \cite[Theorem A.3.5]{Fu94}. Note that
\begin{eqnarray*}
N^{[u]}_{t\wedge\tau_{E_n}}=N^{k,[u_kf_{k}]}_{t\wedge\tau_{E_n}}
+(M^{k,[u_kf_{k}],d}_{t\wedge\tau_{E_n}}-M^{[u],d}_{t\wedge\tau_{E_n}}+A_{t}-A^{p}_{t})+A^{p}_{t}.
\end{eqnarray*}
Hence
$M^{k,[u_kf_{k}],d}_{t\wedge\tau_{E_n}}-M^{[u],d}_{t\wedge\tau_{E_n}}+A_{t}-A^{p}_{t}$
is a purely discontinuous martingale with zero jump, which must be
equal to zero. Consequently,  $N^{[u]}_{t\wedge\tau_{E_n}}$ has
zero quadratic variation w.r.t. $P_{x}$ for ${\cal
E}{\textrm{-}q.e.}\ x\in E$.

Finally, we prove the uniqueness of decomposition (\ref{new3}).
Suppose that $M^1\in{\mathcal{M}}^{I(\zeta)}_{loc}$ and
$N^1\in{\mathcal{L}}_{c}$ such that
$$
\tilde{u}(X_{t})-\tilde{u}(X_{0})=M^1_{t}+N^1_{t},\ \ t\ge0,\ \
P_{x}{\textrm{-}a.s.}\ \ {\rm for}\ {\cal E}{\textrm{-}q.e.}\ x\in
E.
$$
By Proposition \ref{prop00} below, we can choose an
$\{E_n\}\in\Theta$ such that $I(\zeta)=\cup_n
[\![0,\tau_{E_n}]\!]$ $ P_{x}{\textrm{-}a.s.}\ \ {\rm for}\ {\cal
E}{\textrm{-}q.e.}\ x\in E$. Then, for each $n\in\mathbb{N}$,
$\{(M^{[u]}-M^1)^{\tau_{E_n}}\}$ is a locally square integrable
martingale and a zero quadratic variation process w.r.t. $P_m$.
This implies that $P_m(<(M^{[u]}-M^1)^{\tau_{E_n}}>_t=0, \forall
t\in[0,\infty))=0$. Consequently by the analog of \cite[Lemma
5.1.10]{Fu94} in the semi-Dirichlet forms setting,
$P_x(<(M^{[u]}-M^1)^{\tau_{E_n}}>_t=0, \forall t\in[0,\infty))=0$
${\rm for}\ {\cal E}{\textrm{-}q.e.}\ x\in E$. Therefore
$M_t^{[u]}=M_t^1$, $0\le t\le\tau_{E_n}$, $P_{x}{\textrm{-}a.s.}\
{\rm for}\ {\cal E}{\textrm{-}q.e.}\ x\in E$. Since $n$ is
arbitrary, we obtain the uniqueness of decomposition (\ref{new3})
up to the equivalence of local AFs.

(b) Let $u\in D({\cal E})_{loc}$ and suppose that the
decomposition (\ref{new3}) holds. We shall show that $u$ satisfies
Condition (S). First,
$M^{[u],d}\in{\mathcal{M}}^{d,I(\zeta)}_{loc}$ implies that there
exist a sequence of increasing stopping times $\{T_n\}$ such that
$\cup_n[\![0,T_n]\!]=I(\zeta)$ and a sequence of $L^2$-martingales
$\{M^n\}$ such that
$(M^{[u],d}1_{I(\zeta)})^{T_n}=(M^n1_{I(\zeta)})^{T_n}$. Hence
$(M^{[u],d})^{T_n}$  is an $L^2$-martingale and its square bracket
equals $\sum_{0<s\leq t\wedge T_n}(u(X_s)-u(X_{s-}))^{2}$ and is a
integrable increasing process. We use $[M^{[u],d}](t,\omega)$ to
denote $(\sum_{0<s\leq
t}(u(X_s(w))-u(X_{s-}(w)))^{2})1_{I(\zeta)}(t,w)$. Then,
$[M^{[u],d}]\in(\mathcal{A}_{loc,0})^{I(\zeta)}$ (cf.
\cite[\S8.3]{HWY}) and is a local AF. Therefore
$<M^{[u],d}>_t=(\int_0^{t}\int_{E_{\Delta}}(\tilde u(X_s)-\tilde
u(y))^2N(X_s,dy)dH_s)1_{I(\zeta)}$ is a PCAF on $I(\zeta)$ and can
be extended to a PCAF by \cite[Remark 2.2]{Chen}. By Proposition
\ref{lem:Revuz}, its Revuz measure
$\mu^{'}_{u}(dx)=\int_{E_{\Delta}}(\tilde u(x)-\tilde
u(y))^2N(x,dy)\mu_H(dx)$ is a smooth measure. Thus
$\mu_{u}(dx)=\int_E(\tilde u(x)-\tilde u(y))^2J(dy,dx)$, which is
controlled by $\mu^{'}_{u}(dx)$, is also a smooth measure. This
implies that $u$ satisfies Condition (S). \hfill\fbox

%%%%%%%%%%%%%%%%%%%%%%%%%%%%%%%%%%%%%%%%%%%%%%%%%%%%%%%%%%%%%%%%%%%%%%%%%%%%%%%%%%%%%%%%%%%%%%%%%%
\section {Remarks on stochastic sets of interval type}
\label{sec:uniqueness} \setcounter{equation}{0}

For the convenience of the reader, we recall first some concepts
and results concerning sets of interval type given  in
\cite[\S8.3]{HWY}. Let $(\Omega,\mathcal{F},P)$ be a complete
probability space with a filtration $\{\mathcal{F}_t\}$ satisfying
the usual condition. A subset $B\subset\Omega\times[0,\infty)$ is
said to be a set of interval type  if there exists a nonnegative
random variable $T$ such that for each $\omega\in\Omega$, the
section $B_{\omega}$ is either $[0,T(\omega)[$ or $[0,T(\omega)]$
and $B_{\omega}\neq\varnothing$. $B$ is called an optional (resp.
predictable) set of interval type, if it is an optional (resp.
predictable) set and is of interval type.

Let $B$ be an optional set of interval type. A stochastic process
$Y$  defined on $B$ is  called a special semi-martingale  on $B$,
denoted by $({\mathcal{S}}_p)^{B}$, if there exist  a sequence of
increasing stopping times $\{T_n\}$ with $T_n\uparrow T$ ($T$ is
the debut of $B^c$), and a sequence of special semi-martingales
$\{Y^n\}$ such that, $\cup_n[\![0,T_n]\!]\supset B$ and for each
$n$ and $t>0$, $(Y1_{B})_{t\wedge T_n}=(Y^n1_{B})_{t\wedge T_n}$.
In the same manner one can define  local martingale on $B$
(denoted by $({\mathcal{M}}_{loc})^{B}$), adapted process with
locally integrable variation on $B$ (denoted by
$({\mathcal{A}}_{loc})^{B}$), and others (cf. \cite[Definition
8.19]{HWY}).

The assertion below, which is referred as Doob-Meyer decomposition
on sets of interval type, was stated  in \cite[Theorem 8.26]{HWY}.
\vskip 0.2cm \noindent{\bf Assertion.}
   Let $B$  be an optional set of interval type  and
$Y\in(\mathcal{S}_p)^B$. Then $Y$ can be uniquely decomposed as:
$Y=M+A$, where $M\in(\mathcal{M}_{loc})^B$ and
$A\in(\mathcal{A}_{loc,0})^B$ is a predictable process (i.e., $A$
is the restriction of a predictable process on $B$.). \vskip 0.2cm
Although the above assertion has been employed by several papers
(including our previous paper \cite{MMS}),  during the course of
our research we observed the following remark.
\begin{rem}
In the above assertion if $B$ is not a predictable set of interval
type, then the uniqueness of the decomposition $Y=M+A$ may fail to
be true.
\end{rem}
\begin{proof} We take just the  counterexample stated in \cite[Remark 8.24]{HWY} to illustrate our remark.
Let $T>0$ be a totally inaccessible time with $P(T<\infty)>0$,
e.g., the first jump time of a Poisson process. We consider the
stochastic interval $B=[\![0,T[\![$. Then $B$ is an optional set
of interval type but not a predictable set.  Let
$A_t:=1_{[\![T,\infty[\![}(t)$ and $\tilde{A}_t$ be its dual
predictable projection. Let $\{Y_t, 0\leq t< T\}$ be the
restriction of $\tilde{A}$ on $B.$ Then we have decomposition
$Y=M+0$ where $M\in(\mathcal{M}_{loc})^B$ is the restriction of
$\tilde{A}-A$ on $B.$ But we have also another decomposition
$Y=0+Y$ where $Y\in(\mathcal{A}_{loc,0})^B$ is the restriction of
$\tilde{A}$ on $B.$ Therefore the decomposition stated in the
above assertion is not unique.
\end{proof}

The above remark reveals that Doob-Meyer decomposition may fail to
be unique on an optional set of interval type. In the same manner,
we observe that the Fukushima type decomposition may fail to be
unique on an optional set of interval type. Note that with the
notation of Theorem \ref{thm3.2}, $[\![0,\zeta[\![$ is an optional
set of interval type but is not necessarily a predictable set.
\begin{rem} In Theorem \ref{thm3.2} if we use ${\mathcal{M}}^{[\![0,\zeta[\![}_{loc}$
instead of ${\mathcal{M}}^{I(\zeta)}_{loc}$ , then the uniqueness
of the decomposition may fail to be true.
\end{rem}
\begin{proof}
We provide below a counterexample to illustrate the remark.
Suppose that we have a decomposition
\begin{eqnarray*}\label{new4}
\tilde{u}(X_{t})-\tilde{u}(X_{0})=M^{[u]}_{t}+N^{[u]}_{t},\ \
t\ge0,\ \ P_{x}{\textrm{-}a.s.}\ \ {\rm for}\ {\cal
E}{\textrm{-}q.e.}\ x\in E,
\end{eqnarray*}
with $M^{[u]}\in{\mathcal{M}}^{[\![0,\zeta[\![}_{loc}$ and
$N^{[u]}\in {\mathcal{L}}_{c},$ and suppose that $\zeta_i=\zeta$
with $P_x(\zeta<\infty )> 0$ $P_{x}{\textrm{-}a.s.}$ for ${\cal
E}{\textrm{-}q.e.}\ x\in E$.  We write $B_t:=1_{\{\zeta \leq t\}}$
(i.e. $B_t=I_{\Delta}(X_t)$) and denote by $\tilde{B}_t$ the dual
predictable projection of $B_t$. Define
$A_t:=\tilde{B}_t1_{\{0\leq t<\zeta\}}.$ Then  it is clear that
$A\in \mathcal{L}_{c}.$ But we have also
$A\in(\mathcal{M}_{loc})^{[\![0,\zeta[\![}$, because
$\{A1_{[\![0,\zeta[\![}\}^{\zeta}=\{(\tilde{B}-B)1_{[\![0,\zeta[\![}\}^{\zeta}$.
Therefore, we have another decomposition:
\begin{eqnarray*}\label{new4}
\tilde{u}(X_{t})-\tilde{u}(X_{0})=(M^{[u]}_{t}-A_t)+(N^{[u]}_{t}+A_t),\
\ t\ge0,\ \ P_{x}{\textrm{-}a.s.}\ \ {\rm for}\ {\cal
E}{\textrm{-}q.e.}\ x\in E,
\end{eqnarray*}
which violates the uniqueness.
\end{proof}
With the above discussion, we see that the existence of a suitable
predictable set of interval type is important for the uniqueness
of the Fukushima type decomposition. Fortunately in Theorem
\ref{thm3.2} we find such a suitable set
$I(\zeta):=[\![0,\zeta[\![\cup[\![\zeta_i]\!].$ In Proposition
\ref{prop00} below we shall provide a proof for the existence and
uniqueness of such $\zeta_i.$ We shall need the following
characterizations for a set of interval type  to be predictable.
For their proofs we refer to \cite{HWY}.
\begin{lem}\label{yan1}(\cite[Theorems 8.18]{HWY})
The following statements are equivalent:

(i) $B$ is a predictable set of interval type.

(ii) $1_B=1_F1_{[\![0,T[\![}+1_{F^c}1_{[\![0, T]\!]},$ where $T$
is a stopping time, $F\in\mathcal{F}_{T-}$ and $T_F>0$ is a
predictable time.

(iii) $B=\cup_n[\![0,T_n]\!]$, where $\{T_n\}$ is an increasing
sequence of stopping times.
\end{lem}

Below we consider a quasi-regular semi-Dirichlet form  $({\cal
E},D({\cal E}))$ on $L^2(E;m).$  Let ${\bf M}=(\Omega,{\cal
F},({\cal F}_t)_{t\ge 0}, (X_t)_{t\ge 0},(P_x)_{x\in E_{\Delta}})$
with lifetime $\zeta$ be the associated  $m$-tight special
standard process.

\begin{prop}\label{prop00}

(i) There exists an $\{{\cal F}_t\}$-stopping time $\zeta_i$  (may
be identically equal to $\infty$) which is the totally
inaccessible part of $\zeta$ w.r.t. $P_{x}$ for ${\cal
E}{\textrm{-}q.e.}\ x\in E$. Such a $\zeta_i$ is unique in the
sense of $P_{x}{\textrm{-}a.s.}$ for ${\cal E}{\textrm{-}q.e.}\
x\in E$.

(ii) Denote by $I(\zeta):=[\![0,\zeta[\![\cup[\![\zeta_i]\!]$.
Then $I(\zeta)$ is a predictable set of interval type,  and there
exists a sequence $\{V_n\}\in\Theta$ such that for any
$\{U_n\}\in\Theta$, $I(\zeta)=\cup_n [\![0,\tau_{V_n\cap U_n}]\!]$
$ P_{x}{\textrm{-}a.s.}\ \ {\rm for}\ {\cal E}{\textrm{-}q.e.}\
x\in E$.
\end{prop}
\begin{proof}  The uniqueness of $\zeta_i$ follows from \cite[Theorem 4.20]{HWY}.  Below we show the existence of $\zeta_i$ and the assertion (ii).
By the local compactification method (cf. \cite[Theorem
3.5]{HC06}, see also \cite[Theorem VI.1.6]{MR92}) in the
semi-Dirichlet forms setting, we may assume without loss of
generality that $(X_t)_{t\ge 0}$ is a Hunt process and $E$ is a
locally compact separable metric space.

We take a fixed sequence $\{V_n\}\in\Theta$ such that each $V_n$
is a relatively compact open set and  $E=\cup_nV_n$. Denote by
$B:=\cup_n [\![0,\tau_{V_n}]\!]$ and $T:= \lim_{n\rightarrow
\infty} \tau_{V_n}.$  Set $F=\{\omega\,|\,T(\omega)<\infty,
(\omega,T(\omega))\in B^c\}.$ By Lemma \ref{yan1}, for each $P_x,$
it holds that $B$ is a predictable set of interval type, $T$ is an
$\{{\cal F}_t\}$-stopping time, $F\in{\cal F}_{T-},$ $T_{F}:=T
I_{F}+(+\infty)I_{F^c}$ is a predictable time, and
$1_B=1_F1_{[\![0,T[\![}+1_{F^c}1_{[\![0,
T]\!]}=1_{[\![0,T[\![}+1_{[\![ T_{F^c}]\!]}.$  Let $\zeta$ be the
lifetime of $(X_t)_{t\ge 0}$, we define
$$
\zeta_i=\zeta_{F^c}:=\zeta I_{F^c}+(+\infty)I_F.
$$
Note that for ${\cal E}{\textrm{-}q.e.}\ x\in E,$  we have
$\tau_{V_n}\uparrow \zeta=T ~P_x{\textrm{-}a.s.}$, therefore
$I(\zeta):=[\![0,\zeta[\![\cup[\![\zeta_i]\!]
=[\![0,T[\![\cup[\![T_{F^c}]\!]=B$ is a predictable set of
interval type.
 Moreover, by the quasi-left continuity of Hunt process and the assumption that $V_n$ has compact closure, we find that for any $n$ and $x\in E$,
$P_x\{S=\tau_{V_n}=\zeta<\infty\}=0$  for any predictable time
$S.$ Hence $\zeta_{i}=T_{F^c}$ is the totally inaccessible part of
$\zeta$ w.r.t. $P_{x}$ for ${\cal E}{\textrm{-}q.e.}\ x\in E$.
Finally, for arbitrary $\{U_n\}\in\Theta,$ we have $\tau_{V_n\cap
U_n}\uparrow \zeta=T ~P_x~a.s.$ for ${\cal E}{\textrm{-}q.e.}\
x\in E.$ Therefore $I(\zeta)=\cup_n [\![0,\tau_{V_n\cap U_n}]\!]$
$ P_{x}{\textrm{-}a.s.}\ \ {\rm for}\ {\cal E}{\textrm{-}q.e.}\
x\in E,$ which completes the proof.
\end{proof}

%%%%%%%%%%%%%%%%%%%%%%%%%%%%%%%%%%%%%%%%%%%%%%%%%%%%%%%%%%%%%%%%%%%%%%%%%%%%%%%%%%%%%%%%%%%%%%%%%%
\section {Transformation formula for MAFs}
\label{sec:transform} \setcounter{equation}{0}

In this section, we give a transformation formula for MAFs. We
adopt the setting of Section \ref{Sec:Fukushima}. Suppose that
$(\mathcal{E},D(\mathcal{E}))$ is a quasi-regular semi-Dirichlet
form on $L^{2}(E;m)$ satisfying Assumption \ref{assum1}. From the
proof of Theorem \ref{thm3.2}, we can see that $M^{[u],c}$ is well
defined whenever $u\in D(\mathcal{E})_{loc}$. Below is the main
result of this section.
\begin{thm}\label{Ito} Suppose that
$(\mathcal{E},D(\mathcal{E}))$ is a quasi-regular semi-Dirichlet
form on $L^{2}(E;m)$ satisfying Assumption \ref{assum1}. Let
$m\in\mathbb{N}$, $\Phi\in C^1(\mathbb{R}^m)$, and
$u=(u_1,u_2,\dots, u_m)$ with $u_i\in D(\mathcal{E})_{loc}$, $1\le
i\le m$. Then $\Phi(u)\in D(\mathcal{E})_{loc}$ and
$$
M^{[\Phi(u)],c}=\sum_{i=1}^m\Phi_{x_i}(u)\cdot M^{[u_i],c}\ {\rm
on}\ I(\zeta),\ \ P_{x}{\textrm{-}a.s.}\ \ {\rm for}\ {\cal
E}{\textrm{-}q.e.}\ x\in  E.
$$
\end{thm}
 The proof of the theorem will be accomplished at the end of this section by employing Theorem \ref{th4.5} below.

We fix a $\{{V_n}\}\in\Theta$ satisfying Assumption \ref{assum1}
and  such that $\widetilde{\hat{h}}$ is bounded on each $V_n$. Let
$X^{V_n}$, $(\mathcal{E}^{V_n},D(\mathcal{E})_{V_n})$,
$\bar{h}_n$, etc. be the same as in Section 2.
 For $u\in D(\mathcal{E})_{V_n,b}$, we denote by $\mu^{(n)}_{<u>}$ the Revuz
 measure of $<M^{n,[u]}>$.
 %(cf. Lemma \ref{thm3.6} and Theorem \ref{thm2.25} in the Appendix).
 For $u,v\in D(\mathcal{E})_{V_n,b}$, we define
\begin{equation}\label{p01}
 \mu^{(n)}_{<u,v>}:=\frac{1}{2}(\mu^{(n)}_{<u+v>}-\mu^{(n)}_{<u>}-\mu^{(n)}_{<v>}).
\end{equation}
Similar to \cite[Lemma 3.1]{MMS}, we can prove the following
lemma.
 \begin{lem}\label{lem4.1}
 Let $u,v,f\in D(\mathcal{E})_{V_n,b}$. Then
$$
 \int_{V_n}\tilde{f}d\mu^{(n)}_{<u,v>}=\mathcal{E}(u,vf)+\mathcal{E}(v,uf)
 -\mathcal{E}(uv,f).
$$
 \end{lem}

For $u\in {D(\mathcal{E})}_{V_n,b}$, we denote by $M^{n,[u],c}$
and $M^{n,[u],d}$ the continuous  and purely discontinuous parts
of $M^{n,[u]}$, respectively; and denote by $\mu^{n,c}_{<u>}$ and
$\mu^{n,d}_{<u>}$ the Revuz
 measures of $<M^{n,[u],c}>$ and $<M^{n,[u],d}>$, respectively. Then
$M^{n,[u]}=M^{n,[u],c}+M^{n,[u],d}$ and
\begin{equation}\label{n00}
\mu^{(n)}_{<u>}=\mu^{n,c}_{<u>}+\mu^{n,d}_{<u>}.
\end{equation}
 Let $(N^{(n)}(x,dy),H^{(n)})$ be a L\'evy system of $X^{V_n}$ and $\nu^{(n)}$ the
 Revuz measure of $H^{(n)}$. Define $K^{(n)}(dx):=N^{(n)}(x,\Delta)\nu^{(n)}(dx)$. Similar to \cite[(5.3.8) and (5.3.10)]{Fu94}, we can show that
  \begin{eqnarray}\label{l00}
<M^{n,[u],d}>_t&=&\left(\sum_{0<s\leq t}(\bigtriangleup M^{n,[u],d}_{s})^2\right)^p\nonumber\\
&=&\int_0^t\int_{V_n\cup\{\Delta\}}(\tilde u(x)-\tilde
u(y))^2N^{(n)}(X^{V_n}_s,\Delta)dH^{(n)}_s
  \end{eqnarray}
  and
    \begin{equation}\label{xc1}
\mu^{n,d}_{<u>}(dx)=\int_{V_n\cup\{\Delta\}}(\tilde u(x)-\tilde
u(y))^2N^{(n)}(x,dy)\nu^{(n)}(dx).
\end{equation}
   For $u,v\in D(\mathcal{E})_{V_n,b}$, we define
\begin{equation}\label{n02}
 \mu^{n,c}_{<u,v>}:=\frac{1}{2}(\mu^{n,c}_{<u+v>}-\mu^{n,c}_{<u>}-\mu^{n,c}_{<v>}),\ \ \mu^{n,d}_{<u,v>}:=\frac{1}{2}(\mu^{n,d}_{<u+v>}-\mu^{n,d}_{<u>}-\mu^{n,d}_{<v>}).
\end{equation}
\vskip 0.4cm
\begin{thm}\label{th4.5}
Let $u,v,w\in D(\mathcal{E})_{V_n,b}$. Then
\begin{eqnarray}\label{m00}
d\mu^{n,c}_{<uv,w>}=\tilde{u}d\mu^{n,c}_{<v,w>}+\tilde{v}d\mu^{n,c}_{<u,w>}.
\end{eqnarray}
\end{thm}
\begin{proof} The argument for the proof of this theorem is similar to that of \cite[Theorem 3.2]{MMS}.
We will only emphasize the differences caused by the jump part.

By quasi-homeomorphism (cf. \cite[Theorem 3.8]{HC06}) and the
polarization identity, (\ref{m00}) holds for $u,v,w\in
D(\mathcal{E})_{V_n,b}$ is equivalent to
\begin{eqnarray}\label{e4.28*}
\int_{V_n} \tilde{f}d\mu^{n,c}_{<u^2,w>}=2\int_{V_n}
\tilde{f}\tilde{u}d\mu^{n,c}_{<u,w>},\ \ \forall f,u,w\in
D(\mathcal{E})_{V_n,b}.
\end{eqnarray}
For $u,w\in D(\mathcal{E})_{V_n,b}$, we define
$$\eta^{(n)}_{u,w}(dx)=\int_{V_n\cup\{\Delta\}}(\tilde{u}(x)-\tilde{u}(y))^2(\tilde{w}(x)-\tilde{w}(y))N^{(n)}(x,dy)\nu^{(n)}(dx).$$
Then, by (\ref{p01})-(\ref{n02}), we find that (\ref{e4.28*}) is
equivalent to
\begin{eqnarray}\label{p00}
\int_{V_n} \tilde{f}d\mu^{(n)}_{<u^2,w>}=2\int_{V_n}
\tilde{f}\tilde{u}d\mu^{(n)}_{<u,w>}+\int_{V_n}\tilde{f}d\eta^{(n)}_{u,w},\
\ \forall f,u,w\in D(\mathcal{E})_{V_n,b}.\ \
\end{eqnarray}

For $k\in \mathbb{N}$, we define $v_k:=kR^{V_n}_{k+1}u$. Then
$v_k\rightarrow u$ in $D(\mathcal{E})_{V_n}$ as
$k\rightarrow\infty$. By Assumption \ref{assum1} and
\cite[Corollary I.4.15]{MR92}, we can show that
$\sup_{k\geq1}\mathcal{E}(v_kw,v_kw)<\infty$. Then, by \cite[Lemma
I.2.12]{MR92}, there exists a subsequence
$\{(v_{k_l})\}_{l\in\mathbb{N}}$ of $\{v_k\}_{k\in\mathbb{N}}$
such that $u_kw\rightarrow uw$ in $D(\mathcal{E})_{V_n}$ as
$k\rightarrow\infty$, where $u_k:=\frac{1}{k}\sum_{l=1}^kv_{k_l}$.
Note that $u_k\rightarrow u$ in $D(\mathcal{E})_{V_n}$ as
$k\rightarrow\infty$  and $\|u_k\|_{\infty}\le \|u\|_{\infty}$ for
$k\in \mathbb{N}$. Moreover, $\|L^{V_n}u_k\|_{\infty}<\infty$ for
$k\in \mathbb{N}$, where $L^{V_n}$ is the generator of $X^{V_n}$.
For $k,l\in \mathbb{N}$, we define $f_{k}:=f\wedge (k\bar{h}_n)$
and $f_{k,l}:=l\hat{G}^{V_n}_{l+1}f_k$.

Similar to \cite[Theorem 3.2]{MMS}, to prove (\ref{p00}), we may
assume without loss of generality that $f\ge 0$, $u=u_k$ and
$f=f_{k,l}$.

For $0<\delta<1$, we have
\begin{eqnarray*}\label{n11}
& &\lim_{t\downarrow0}\frac{1}{t}E_{f_{k,l}\cdot m}[<M^{n,[u_k]}>^2_t]\nonumber\\
&&\ \ \ \ =\lim_{t\downarrow0}\frac{2}{t}E_{f_{k,l}\cdot m}\left[\int_0^t<M^{n,[u_k]}>_{(t-s)}\circ\theta_s d<M^{n,[u_k]}>_s\right]\nonumber\\
&&\ \ \ \ =\lim_{t\downarrow0}\frac{2}{t}E_{f_{k,l}\cdot m}\left[\int_0^tE_{X^{V_n}_s}[<M^{n,[u_k]}>_{(t-s)}]d<M^{n,[u_k]}>_s\right]\nonumber\\
&&\ \ \ \
\le2<E_{\cdot}[<M^{n,[u_k]}>_{\delta}]\cdot\mu^{(n)}_{<u_k>},\widetilde{f_{k,l}}>.
\end{eqnarray*}
Note that by our choice of $u_k$, there exists a constant $D_k>0$
such that
$E_x(<M^{n,[u_k]}>_{\delta})=E_x[(M^{n,[u_k]}_{\delta})^2]=
E_x[(\widetilde{u_k}(X^{V_n}_{\delta})-\widetilde{u_k}(X^{V_n}_0)-\int_0^{\delta}L^{V_n}u_k(X^{V_n}_s)ds)^2]\le
D_k$ for ${\cal E}{\textrm{-}q.e.}\ x\in V_n$. Letting
$\delta\rightarrow 0$, we obtain by the dominated convergence
theorem that \begin{equation}\label{addfg}
\lim_{t\downarrow0}\frac{1}{t}E_{f_{k,l}\cdot
m}[<M^{n,[u_k]}>^2_t]=0.
\end{equation}

We have
\begin{eqnarray*}
\int_{V_n} \widetilde{f_{k,l}}d\mu^{(n)}_{<u_k^2,w>}&=&\lim_{t\downarrow0}{1\over t}E_{f_{k,l}\cdot m}[<M^{n,[u_k^2]},M^{n,[w]}>_t]\nonumber\\
&=&\lim_{t\downarrow0}\frac{1}{t}E_{f_{k,l}\cdot m}[(\widetilde{u_k}^2(X^{V_n}_t)-\widetilde{u_k}^2(X^{V_n}_0))(\tilde{w}(X^{V_n}_t)-\tilde{w}(X^{V_n}_0))]\nonumber\\
&=&\lim_{t\downarrow0}\frac{2}{t}E_{(f_{k,l}u_k)\cdot m}[(\widetilde{u_k}(X^{V_n}_t)-\widetilde{u_k}(X^{V_n}_0))(\tilde{w}(X^{V_n}_t)-\tilde{w}(X^{V_n}_0))]\nonumber\\
  &&+\lim_{t\downarrow0}\frac{1}{t}E_{f_{k,l}\cdot m}[(\widetilde{u_k}(X^{V_n}_t)-\widetilde{u_k}(X^{V_n}_0))^2(\tilde{w}(X^{V_n}_t)-\tilde{w}(X^{V_n}_0))]\nonumber\\
&:=&\lim_{t\downarrow0}[{I(t)+II(t)}].
\end{eqnarray*}

Similar to \cite[Theorem 3.2]{MMS}, we can show that
$$
\lim_{t\downarrow0}I(t)=2\int_{V_n}\widetilde{f_{k,l}}\widetilde{u_k}d\mu^{(n)}_{<u_k,w>}.
$$

Note that
\begin{eqnarray*}
\lim_{t\downarrow0}II(t)
&=&\lim_{t\downarrow0}\frac{1}{t}E_{f_{k,l}\cdot m}[(M^{n,[u_k],c}_t)^2M^{n,[w]}_t]\nonumber\\
&&+2\lim_{t\downarrow0}\frac{1}{t}E_{f_{k,l}\cdot m}[(M^{n,[u_k],c}_t)(M^{n,[u_k],d}_t)M^{n,[w],c}_t]\nonumber\\
&&+2\lim_{t\downarrow0}\frac{1}{t}E_{f_{k,l}\cdot m}[(M^{n,[u_k],c}_t)(M^{n,[u_k],d}_t)M^{n,[w],d}_t]\nonumber\\
&&+\lim_{t\downarrow0}\frac{1}{t}E_{f_{k,l}\cdot m}[(M^{n,[u_k],d}_t)^2M^{n,[w],c}_t]\nonumber\\
  &&+\lim_{t\downarrow0}\frac{1}{t}E_{f_{k,l}\cdot m}[(M^{n,[u_k],d}_t)^2M^{n,[w],d}_t]\nonumber\\
   &:=&\lim_{t\downarrow0}\{III_{1}(t)+2III_{2}(t)+2III_{3}(t)+III_{4}(t)+IV(t)\}.
  \end{eqnarray*}
Similar to \cite[Theorem 3.2]{MMS}, we can show that
\begin{equation}\label{q111}\lim_{t\downarrow0}III_{1}(t)=0.
\end{equation}
By It$\hat{\rm o}$'s formula and the orthogonality of the
continuous and purely discontinuous martingales, we get
\begin{eqnarray*}
\lim_{t\downarrow0}|III_{2}(t)|&\leq&\left\{\lim_{t\downarrow0}\frac{1}{t}E_{f_{k,l}\cdot m}[<M^{n,[u_k],c},M^{n,[w],c}>^{2}_{t}]\right\}^{\frac{1}{2}}\\
& &\ \ \cdot\left\{\lim_{t\downarrow0}\frac{1}{t}E_{f_{k,l}\cdot
m}[M^{n,[u_k],d}_t]^{2}\right\}^{\frac{1}{2}}.
\end{eqnarray*}
Similar to (\ref{q111}), we can show that
$\lim_{t\downarrow0}III_{2}(t)=0$.

By It$\hat{\rm o}$'s formula and Burkholder-Davis-Gundy
inequality, we get
\begin{eqnarray*}
\lim_{t\downarrow0}|III_{4}(t)|&=&\lim_{t\downarrow0}\left|\frac{1}{t}E_{f_{k,l}\cdot m}\left\{\sum_{0<s\leq t}M^{n,[w],c}_s(\bigtriangleup M^{n,[u_k],d}_s)^2\right\}\right|\\
&&=\lim_{t\downarrow0}\left|\frac{1}{t}E_{f_{k,l}\cdot m}\left\{\int_{0}^{t}M^{n,[w],c}_sd<M^{n,[u_k],d}>_s\right\}\right|\\
&&\leq\lim_{t\downarrow0}\frac{1}{t}E_{f_{k,l}\cdot m}\{M^{n,[w],c \ast}_t<M^{n,[u_k],d}>_t\}\\
&&\leq\left\{\lim_{t\downarrow0}\frac{1}{t}E_{f_{k,l}\cdot m}(M^{n,[w],c \ast}_t)^2\right\}^{\frac{1}{2}}\left\{\lim_{t\downarrow0}\frac{1}{t}E_{f_{k,l}\cdot m}(<M^{n,[u_k],d}>_t)^2\right\}^{\frac{1}{2}}\\
&&\leq C\left\{\lim_{t\downarrow0}\frac{1}{t}E_{f_{k,l}\cdot m}(M^{n,[w],c}_t)^2\right\}^{\frac{1}{2}}\left\{\lim_{t\downarrow0}\frac{1}{t}E_{f_{k,l}\cdot m}(<M^{n,[u_k],d}>_t)^2\right\}^{\frac{1}{2}}\\
&&=C\left\{\lim_{t\downarrow0}\frac{1}{t}E_{f_{k,l}\cdot
m}<M^{n,[w],c}_t>\right\}^{\frac{1}{2}}\left\{\lim_{t\downarrow0}\frac{1}{t}E_{f_{k,l}\cdot
m}(<M^{n,[u_k],d}>_t)^2\right\}^{\frac{1}{2}},
\end{eqnarray*}
where $M^{n,[w],c \ast}_t$ denotes the maximum of $M^{n,[w],c}_t$,
$\bigtriangleup
M^{n,[u_k],d}_s=M^{n,[u_k],d}_s-M^{n,[u_k],d}_{s-}$and $C$  is a
positive constant. Hence $\lim_{t\downarrow0}III_{4}(t)=0$.
Similarly, we can show that $\lim_{t\downarrow0}III_{3}(t)=0$.

Finally, we estimate $IV(t)$. By It$\hat{\rm o}$'s formula and the
dual predictable projection, we get
\begin{eqnarray*}
IV(t)&=&\frac{1}{t}E_{f_{k,l}\cdot m}(M^{n,[u_k],d}_t)^2M^{n,[w],d}_t\\
&=&\frac{1}{t}E_{f_{k,l}\cdot m}\left\{\sum_{0<s\leq t}(M^{n,[u_k],d}_s)^2M^{n,[w],d}_s-(M^{n,[u_k],d}_{s-})^2M^{n,[w],d}_{s-}\right.\\
&&\left.-2M^{n,[u_k],d}_{s-}M^{n,[w],d}_{s-}(M^{n,[u_k],d}_s-M^{n,[u_k],d}_{s-})-
(M^{n,[u_k],d}_{s-})^2(M^{n,[w],d}_s-M^{n,[w],d}_{s-})\right\}\\
&=&\frac{1}{t}E_{f_{k,l}\cdot m}\left\{\sum_{0<s\leq t}(\bigtriangleup M^{n,[u_k],d}_s)^2\bigtriangleup M^{n,[w],d}_s\right.\\
&&\left.+\sum_{0<s\leq t}M^{n,[w],d}_{s-}(\bigtriangleup M^{n,[u_k],d}_s)^2+\sum_{0<s\leq t}M^{n,[u_k],d}_{s-}\bigtriangleup M^{n,[u_k],d}_s\bigtriangleup M^{n,[w],d}_s\right\}\\
&=&\frac{1}{t}E_{f_{k,l}\cdot m}\left\{\int_0^t\int_{V_n\cup\{\Delta\}}(u_k(X^{V_n}_s)-u_k(y))^2(w(X^{V_n}_s)-w(y))N^{(n)}(X^{V_n}_s,dy)dH^{(n)}_s\right.\\
&&\left.+\sum_{0<s\leq t}M^{n,[w],d}_{s-}(\bigtriangleup M^{n,[u_k],d}_s)^2+\sum_{0<s\leq t}M^{n,[u_k],d}_{s-}\bigtriangleup M^{n,[u_k],d}_s\bigtriangleup M^{n,[w],d}_s\right\}\\
&:=&IV_1(t)+IV_2(t)+IV_3(t).
\end{eqnarray*}
We have
\begin{eqnarray*}
\lim_{t\downarrow0}IV_1(t)=\int_{V_n}f_{k,l}d\eta^{(n)}_{u_k,w}
\end{eqnarray*}
and, by Lemma \ref{thm3.6} and (\ref{addfg}),
\begin{eqnarray*}
\lim_{t\downarrow0}|IV_2(t)|&=&\lim_{t\downarrow0}\left|\frac{1}{t}E_{f_{k,l}\cdot m}\left\{\int^{t}_{0}M^{n,[w],d}_{s-}d<M^{n,[u_k],d}>_{s}\right\}\right|\\
&\leq&\lim_{t\downarrow0}\frac{1}{t}E_{f_{k,l}\cdot m}\{(M^{n,[w],d}_{t})^{*}<M^{n,[u_k],d}>_{t}\}\\
&\leq&\left\{\lim_{t\downarrow0}\frac{1}{t}E_{f_{k,l}\cdot m}<M^{n,[w],d}>_t\right\}^{\frac{1}{2}}\left\{\lim_{t\downarrow0}\frac{1}{t}E_{f_{k,l}\cdot m}<M^{n,[u_k],d}>_{t}^2\right\}^{\frac{1}{2}}\\
&=&0,
\end{eqnarray*}
where $M^{n,[w],d \ast}_t$ denotes the maximum of $M^{n,[w],d}_t$.
Similarly, we get $\lim_{t\downarrow0}IV_3(t)=0$. Therefore, the
proof is complete.
\end{proof}
\noindent{\bf Proof of Theorem \ref{Ito}} By virtue of Theorem
\ref{th4.5}, following the argument of the proof of \cite[Theorem
3.10]{MMS}, we can prove Theorem \ref{Ito}. We omit the details
here.

\hfill\fbox

\section[short
title]{Examples}\label{Sec:example}

In this section, we consider some concrete examples. Note that our
Theorems \ref{thm3.2} and \ref{Ito} are generalization of the
corresponding results of \cite{MMS}, which were only given for
local semi-Dirichlet forms without jump.

\begin{exa}\label{e3}(see \cite{Fu12} and cf. also \cite{SW}) Let $(E,d)$ be a locally compact separable metric space,
m a positive Radon Measure on $E$ with full topological support,
and $k(x,y)$  a nonnegative  Borel measurable function on
$\{(x,y)\in E\times E\,|\,x\not=y \}$. Set
$k_{s}(x,y)=\frac{1}{2}(k(x,y)+k(y,x))$ and
$k_{a}(x,y)=\frac{1}{2}(k(x,y)-k(y,x)) $. Denote by
$C^{lip}_{0}(E)$ the family of all uniformly Lipschitz continuous
functions on $E$ with compact support. Suppose that the following
conditions hold:

\noindent (A.I) $x\rightarrow\int_{y\neq x}(1\wedge
d(x,y)^{2})k_{s}(x,y)m(dy)\in L^{1}_{loc}(E;m)$.

\noindent (A.II) $\sup_{x\in E}\int_{\{y:\,k_s(x,y)\not=0\}}
\frac{k_a^2(x,y)}{k_s(x,y)}m(dy)<\infty$.

Define for $u,v\in C^{lip}_{0}(E)$,
$$
\eta(u,v)=\int\hskip-0.2cm\int_{x\not=y}(u(x)-u(y))(v(x)-v(y))k_s(x,y)m(dx)m(dy)
$$
and
$$
{\cal
E}(u,v)=\frac{1}{2}\eta(u,v)+\int\hskip-0.2cm\int_{x\not=y}(u(x)-u(y))v(y)k_a(x,y)m(dx)m(dy).$$
Then, there exists $\alpha>0$ such that $({\cal
E}_{\alpha},C^{lip}_{0}(E))$ is closable on $L^2(E;m)$ and its
closure $({\cal E}_{\alpha},D({\cal E}_{\alpha}))$ is a regular
semi-Dirichlet form on $L^2(E;m)$. Moreover, there exists $C>1$
such that for any $u\in D(\mathcal{E}_{\alpha})$,
\begin{eqnarray*}
\frac{1}{C}\eta_{\alpha}(u,u)\leq {\cal E}_{\alpha}(u,u)\leq
C\eta_{\alpha}(u,u).
\end{eqnarray*}
Therefore, our Theorems \ref{thm3.2} and \ref{Ito} hold for any
$u\in {D({\cal E})}_{loc}$ which satisfies Condition (S), in
particular, for any $u\in {D({\cal E})}$ by noting that
$|k_a(x,y)|\le k_s(x,y)$.
\end{exa}

\begin{exa}\label{e3}(see \cite{U}) Let $G$ be an open set of $\mathbb{R}^d$. Suppose that the following
conditions hold:

\noindent (B.I) There exist $0<\lambda\le \Lambda$ such that
$$
\lambda|\xi|^2\le \sum_{i,j=1}^da_{ij}(x)\xi_i\xi_j\le
\Lambda|\xi|^2\ \ {\rm for}\ x\in G,\ \xi\in\mathbb{R}^d.
$$

\noindent (B.II) $b_i\in L^{d}(G;dx)$, $i=1,\dots,d$.

\noindent (B.III) $c\in L^{d/2}_+(G;dx)$.

\noindent (B.IV)
$x\rightarrow\int_{y\not=x}(1\wedge|x-y|^2)k_s(x,y)dy\in
L^1_{loc}(G;dx)$.

\noindent (B.V) $ \sup_{x\in G}\int_{\{|x-y|\ge 1,y\in
G\}}|k_a(x,y)|dy<\infty$, $ \sup_{x\in G}\int_{\{|x-y|< 1,y\in
G\}}|k_a(x,y)|^{\gamma}dy<\infty$ for some $0<\gamma\le 1$, and
$|k_a(x,y)|^{2-\gamma}\le C_1k_s(x,y)$, $x,y\in G$ with
$0<|x-y|<1$ for some constant $C_1>0$.

Define for $u,v\in C^1_{0}(G)$,
\begin{eqnarray*}
\eta(u,v)&=&\frac{1}{2}\sum_{i=1}^d\int_G\frac{\partial
u}{\partial x_i}(x)\frac{\partial v}{\partial
x_i}(x)dx\\
& &\ \ \ \
+\frac{1}{2}\int\hskip-0.2cm\int_{x\not=y}(u(x)-u(y))(v(x)-v(y))k_s(x,y)dxdy
\end{eqnarray*}
and
\begin{eqnarray*}
{\cal
E}(u,v)&=&\frac{1}{2}\sum_{i=1}^d\int_Ga_{ij}(x)\frac{\partial
u}{\partial x_i}(x)\frac{\partial v}{\partial
x_j}(x)dx+\sum_{i=1}^d\int_Gb_{i}(x)u(x)\frac{\partial v}{\partial
x_i}(x)dx\\
& &\ \ \ \ +\int_Gu(x)v(x)c(x)dx\\
& &\ \ \ \ +
\frac{1}{2}\int\hskip-0.2cm\int_{x\not=y}(u(x)-u(y))(v(x)-v(y))k_s(x,y)dxdy\\
& &\ \ \ \
+\int\hskip-0.2cm\int_{x\not=y}(u(x)-u(y))v(x)k_a(x,y)dxdy.
\end{eqnarray*}
Then, when $\lambda$ is sufficiently large, there exists
$\alpha>0$ such that $({\cal E}_{\alpha},C^1_{0}(G))$ is closable
on $L^2(G;dx)$ and its closure $({\cal E}_{\alpha},D({\cal
E}_{\alpha}))$ is a regular semi-Dirichlet form on $L^2(G;dx)$.
Moreover, there exists $C'>1$ such that for any $u\in
D(\mathcal{E}_{\alpha})$,
\begin{eqnarray*}
\frac{1}{C'}\eta_{\alpha}(u,u)\leq {\cal E}_{\alpha}(u,u)\leq
C'\eta_{\alpha}(u,u).
\end{eqnarray*}
Therefore, our Theorems \ref{thm3.2} and \ref{Ito} hold for any
$u\in {D({\cal E})}_{loc}$ which satisfies Condition (S), in
particular, for any $u\in {D({\cal E})}$ by noting that
$|k_a(x,y)|\le k_s(x,y)$.
\end{exa}

\bigskip
%%%%%%%%%%%% Authors' addresses %%%%%%%%%%%%%

\noindent \textbf{Author information} \vskip 0.2cm
\noindent Zhi-Ming Ma, Academy of Mathematics and Systems Science \\
Chinese Academy of Sciences \\
No.55, Zhong-guan-cun East Road,\\
Beijing, 100190, China \\
             E-mail: mazm@amt.ac.cn

\vskip 0.2cm
\noindent Wei Sun, Department of Mathematics and Statistics \\
Concordia University \\
Montreal, H3G 1M8, Canada \\
             E-mail: wei.sun@concordia.ca

\vskip 0.2cm
\noindent Li-Fei Wang, College of Mathematics and Information Science \\
Hebei Normal University \\
Shijiazhuang, 050024, China\\
             E-mail: flywit1986@163.com

\end{document}